\documentclass{amsart}

\usepackage{amssymb,latexsym}
\usepackage{graphicx}
\usepackage[all,2cell]{xy}
\SelectTips{eu}{12}

\newcommand{\Z}{\mathbb{Z}}

\newcommand{\Q}{\mathbb{Q}}
\newcommand{\R}{\mathbb{R}}

\newcommand{\SO}{\mathrm{SO}}
\newcommand{\OO}{\mathrm{O}}

\newcommand{\K}{\mathbf{k}}
\newcommand{\F}{\mathbf{F}}

\newtheorem{lemma}{Lemma}[section]

\newtheorem{proposition}[lemma]{Proposition}
\newtheorem{theorem}[lemma]{Theorem}
\newtheorem{corollary}[lemma]{Corollary}

\theoremstyle{definition}
\newtheorem{remark}[lemma]{Remark}
\newtheorem{definition}[lemma]{Definition}
\newtheorem{example}[lemma]{Example}

\begin{document}
\parindent0em
\setlength\parskip{.1cm}
\title{Intersection homology of linkage spaces}
\author{Dirk Sch\"utz}
\address{Department of Mathematical Sciences\\ University of Durham\\ Science Laboritories\\ South Rd\\ Durham DH1 3LE\\ United Kingdom}
\email{dirk.schuetz@durham.ac.uk}
\begin{abstract}
We consider the moduli spaces $\mathcal{M}_d(\ell)$ of a closed linkage with $n$ links and prescribed lengths $\ell\in \R^n$ in $d$-dimensional Euclidean space. For $d>3$ these spaces are no longer manifolds generically, but they have the structure of a pseudomanifold.

We use intersection homology to assign a ring to these spaces that can be used to distinguish the homeomorphism types of $\mathcal{M}_d(\ell)$ for a large class of length vectors in the case of $d$ even. This result is a high-dimensional analogue of the Walker conjecture which was proven by Farber, Hausmann and the author.
\end{abstract}
\maketitle

\section{Introduction}
Configuration spaces of closed linkages in Euclidean space modulo isometry group have occured in many contexts in recent years. Planar linkages can easily be visualised and the topology of the resulting moduli spaces are now well understood, culminating in the proof of the Walker conjecture by Farber, Hausmann and the author in \cite{fahasc,schue2}. Roughly this conjecture states that the cohomology of the linkage space detects the length vector of the linkage. By a \em length vector \em we simply mean an element $\ell=(\ell_1,\ldots,\ell_n)\in \R^n$ such that all entries are positive. The $i$-th entry $\ell_i$ describes the length of the $i$-th link.

For linkages in 3-dimensional Euclidean space, the resulting moduli spaces occur naturally in algebraic geometry and symplectic geometry, see e.g.\ \cite{klyach, kapmil}, and the cohomology rings have been calculated in Hausmann and Knutson \cite{hauknu}. Using this description of cohomology, the analogue of the Walker conjecture was proven in \cite{fahasc}, with the single exception that for $n=4$ there exist two different length vectors whose moduli spaces are both the 2-sphere.

Much less is known for linkages in higher-dimensional Euclidean spaces. For linkages in 5-dimensional space, Foth and Lozano obtained analogous results of Kapovich and Millson \cite{kapmil} in a quaternion setting rather than a complex one. Kamiyama \cite{kamiya} obtained an Euler characteristic formula for equilateral linkages in 4-dimensional space, and more recently homology calculations were obtained in \cite{schuet}.

The moduli space we are interested in is
\begin{eqnarray*}
 \mathcal{M}_d(\ell)&=&\left\{ (x_1,\ldots,x_n)\in (S^{d-1})^n\,\left|\, \sum_{i=1}^n \ell_ix_i=0\right\}\right/ \SO(d)
\end{eqnarray*}
where $\SO(d)$ acts diagonally on the product of spheres. In particular we want to know how the topology of $\mathcal{M}_d(\ell)$ depends on the length vector $\ell$. Permuting the coordinates of $\ell$ does not change the topology as we can simply permute the coordinates of $\mathcal{M}_d(\ell)$. It also turns out that small changes of $\ell$ do not change the topology, provided that $\ell$ does not admit a collinear configuration, that is, if $\mathcal{M}_1(\ell)=\emptyset$. If this is the case, we call the length vector \em generic\em.

Indeed, the non-generic length vectors are the boundaries of so-called chambers, connected open subsets of $\R^n$ such that any two length vectors in the same chamber admit homeomorphic moduli spaces.

In general, if two length vectors $\ell$, $\ell'$ are in different chambers, even after permuting coordinates, it does not necessarily follow that $\mathcal{M}_d(\ell)$ and $\mathcal{M}_d(\ell')$ are not homeomorphic. In fact, Schoenberg \cite{schoen} showed that for $d=n$ the moduli space $\mathcal{M}_d(\ell)$ is either a topological disc or empty, from which it can be seen that for $d=n-1$ the moduli spaces are empty or topologically a sphere. The case where the moduli space is empty is represented by the chamber where one coordinate $\ell_j$ is bigger than the sum of all other $\ell_i$, so in all other cases we always get the same moduli space. Notice that the case $d>n$ reduces to the case $d=n$, as the extra dimensions in $\R^d$ cannot be taken advantage of by linear dependence of the $x_1,\ldots,x_n$.

In the case $d<n-1$ the topology of the moduli space does depend on the chamber, as can be seen from the homology calculations in \cite{schuet}. The main result of this paper shows that for a large class of length vectors the topology of the moduli space does recover the chamber of the length vector. 

\begin{theorem}
 \label{main_theorem}
Let $d\geq 4$ be even, $\ell,\ell'\in \R^n$ be generic, $d$-regular length vectors. If $\mathcal{M}_d(\ell)$ and $\mathcal{M}_d(\ell')$ are homeomorphic, then $\ell$ and $\ell'$ are in the same chamber up to a permutation.
\end{theorem}

The notion of $d$-regular is defined in Section \ref{section_linkspaces}, in view of Schoenberg's result it should be pointed out that for $n=d+1$ there are exactly two chambers up to permutation which contain $d$-regular length vectors, one with empty moduli space and one where the moduli space is a sphere. If $n$ is large compared to $d$, $d$-regularity is more common, and we would expect the ratio of all $d$-regular length vectors in $\R^n$ by all length vectors in $\R^n$ to converge to $1$.

The statement of the theorem is known to be true for $d=2$, as it follows from the proof of the Walker conjecture in \cite{fahasc,schue2}, and for $d=3$, as was shown in \cite{fahasc}. Every generic length vector is $2$-regular, and there is only one chamber up to permutation so that its length vectors are not $3$-regular.

Homology calculations are not enough to obtain Theorem \ref{main_theorem}, and in fact the cases $d=2,3$ were obtained using cohomology. If we do not form the quotient by $\SO(d)$ and look instead at a configuration space $E_d(\ell)$ (so that $\mathcal{M}_d(\ell)=E_d(\ell)/\SO(d)$), cohomology is again enough to detect the chamber of $\ell$, see Farber and Fromm \cite{farfro}. 

It is clear from the calculations in \cite{schuet} that ordinary cohomology is not enough for $d\geq 4$. Instead we use intersection homology in this paper. By letting the perversity vary with the degree of the intersection homology group, we can use the intersection pairing to assign a ring to each moduli space which behaves very similar to the cohomology ring in the case $d=2$. For even $d\geq 4$ we can explicitely describe this ring and use it to prove Theorem \ref{main_theorem}.

The idea of the ring based on intersection homology is roughly the following. Given $J\subset \{1,\ldots,n\}$ with $n\in J$ we can form a new length vector $\ell_J\in \R^{n-|J|+1}$ by merging the links corresponding to the elements of $J$ into one link. This leads to an inclusion $\mathcal{M}_d(\ell_J)\subset \mathcal{M}_d(\ell)$. Furthermore, if $K\cap J=\{n\}$, the intersection $\mathcal{M}_d(\ell_J)\cap \mathcal{M}_d(\ell_K)=\mathcal{M}_d(\ell_{J\cup K})$ is transversal. Each $J$ has a perversity $\mathbf{p}_{|J|}$ such that we have an element $[\mathcal{M}_d(\ell_J)]\in I^{\mathbf{p}_{|J|}}H_\ast(\mathcal{M}_d(\ell))$ so that these elements behave well with the intersection pairing of Goresky-MacPherson \cite{gormac, gorma2}. For even $d$, these elements turn out to span an exterior algebra which is invariant under homeomorphism. The condition of $d$-regularity ensures that $\mathcal{M}_d(\ell_J)$ is not a disc in which case it would be invisible for homology.

We expect Theorem \ref{main_theorem} to be true for odd $d\geq 5$ and it may be possible to prove it using the intersection ring defined in this paper, however, the actual determination of this ring for odd $d$ will probably require new techniques.

\section{Linkage spaces and intersection homology}
\label{section_linkspaces}
In order to study $\mathcal{M}_d(\ell)$ it is useful to consider the \em chain space\em
\begin{eqnarray*}
 \mathcal{C}_d(\ell)&=&\left\{(x_1,\ldots,x_{n-1})\in (S^{d-1})^{n-1}\,\left|\, \sum_{i=1}^{n-1}\ell_i x_i=-\ell_ne_1\right\}\right.
\end{eqnarray*}
where $e_1=(1,0,\ldots,0)\in \R^d$ is the usual first coordinate vector. If we let $\SO(d-1)$ act on $S^{d-1}$ by fixing the first coordinate, we see that $\SO(d-1)$ acts diagonally on $\mathcal{C}_d(\ell)$ and
\begin{eqnarray*}
 \mathcal{M}_d(\ell)&\cong & \mathcal{C}_d(\ell)/\SO(d-1).
\end{eqnarray*}
We also define
\begin{eqnarray*}
 \mathcal{N}_d(\ell)&=&\left\{ (x_1,\ldots,x_n)\in (S^{d-1})^n\,\left|\, \sum_{i=1}^n \ell_ix_i=0\right\}\right/ \OO(d)
\end{eqnarray*}
so that $\mathcal{N}_d(\ell)\cong \mathcal{M}_d(\ell)/(\Z/2)$.

\begin{definition}
 Let $\ell\in \R^n$ be a length vector. A subset $J\subset \{1,\ldots,n\}$ is called \em $\ell$-short\em, if
\begin{eqnarray*}
 \sum_{j\in J}\ell_j&<& \sum_{i\notin J}\ell_i.
\end{eqnarray*}
It is called \em $\ell$-long\em, if the complement is $\ell$-short, and \em $\ell$-median\em, if it is neither $\ell$-short nor $\ell$-long. The length vector is called \em generic\em, if there are no $\ell$-median subsets.

For $m\in \{1,\ldots,n\}$ the length vector is called \em $m$-dominated\em, if $\ell_m\geq \ell_i$ for all $i=1,\ldots,n$.
\end{definition}

After permuting the coordinates we can always assume that $\ell$ is $n$-dominated. In fact, we can also assume that $\ell$ is \em ordered\em, meaning that $\ell_1\leq \ell_2\leq\cdots\leq \ell_n$.

If $\ell$ is $m$-dominated and $k\leq n-3$, we write
\begin{eqnarray*}
 \mathcal{S}_k(\ell)&=& \{J\subset\{1,\ldots,n\}\,|\,m\in J, \, |J|=k+1, J \mbox{ is }\ell\mbox{-short}\}.
\end{eqnarray*}

Note that a length vector $\ell$ can be $m$-dominated by more than one $m\in \{1,\ldots,n\}$. In this case we will form $\mathcal{S}_k(\ell)$ using the maximal $m$ which dominates $\ell$.

If $J\subset\{1,\ldots,n\}$, we define the hyperplane
\begin{eqnarray*}
 H_J&=&\left\{(x_1,\ldots,x_n)\in \R^n\,\left|\, \sum_{j\in J}x_j= \sum_{j\notin J}x_j\right\}\right.
\end{eqnarray*}
and let
\begin{eqnarray*}
 \mathcal{H}&=& \R^n_{>0} - \bigcup_{J\subset\{1,\ldots,n\}}H_J,
\end{eqnarray*}
where $\R^n_{>0}=\{(x_1,\ldots,x_n)\in \R^n\,|\,x_i>0\}$. Then $\mathcal{H}$ has finitely many components, which we call \em chambers\em. It is clear that a length vector $\ell$ is generic if and only if $\ell\in\mathcal{H}$. 

It is shown in \cite{hausma} that if $\ell$ and $\ell'$ are in the same chamber, then $\mathcal{C}_d(\ell)$ and $\mathcal{C}_d(\ell')$ are $\OO(d-1)$-equivariantly diffeomorphic. In particular, $\mathcal{M}_d(\ell)$ and $\mathcal{M}_d(\ell')$ are homeomorphic.

It is easy to see that two $m$-dominated generic length vectors $\ell$, $\ell'$ are in the same chamber if and only if $\mathcal{S}_k(\ell)=\mathcal{S}_k(\ell')$ for all $k=0,\ldots,n-3$.

\begin{definition}
 Let $\ell\in \R^n$ be a length vector and $d\geq 2$. Then $\ell$ is called $d$-regular, if
\begin{eqnarray*}
 \bigcap_{J\in \mathcal{L}^d(\ell)}J &\not=& \emptyset
\end{eqnarray*}
where $\mathcal{L}^d(\ell)$ are the subsets $J\subset\{1,\ldots,n\}$ with $d-1$ elements that are $\ell$-long. If $\mathcal{L}^d(\ell)=\emptyset$, we let the intersection above be $\{1,\ldots,n\}$.
\end{definition}

If $\ell$ is ordered, then $\ell$ is $d$-regular if and only if $\{n-d+1,n-d+2,\ldots,n-1\}$ is not $\ell$-long. For a generic length vector this is equivalent to $\mathcal{S}_{n-d}(\ell)=\emptyset$.

It follows from the definition that every length vector with $n\geq 2$ is $2$-regular. Furthermore, there is only one generic length vector up to permutation which is not $3$-regular, namely $\ell=(0,\ldots,0,1,1,1)$. In \cite{fahasc}, $4$-regular was called normal.

In the case $d=n-1$, there are only two generic length vectors $\ell\in \R^n$ up to permutation, namely $\ell=(1,\ldots,1,n-2)$ and $\ell'=(0,\ldots,0,1)$. If $n$ is large compared to $d$, $d$-regularity gets more common, and we would expect the ratio of all $d$-regular length vectors in $\R^n$ by all length vectors in $\R^n$ to converge to $1$.

For $d=2,3$ and $\ell$ generic, the spaces $\mathcal{M}_d(\ell)$ are closed manifolds, but for $d\geq 4$ this is no longer the case. But as we will see in Section \ref{section_pseudo}, these spaces are pseudomanifolds for $n>d$. For the precise definition of a pseudomanifold, we refer the reader to \cite{gorma2}. Since we need intersection homology below, we do recall some of the definitions in order to agree on notation. Given an $n$-dimensional pseudomanifold $X^n$, there is a stratification
\[
 \emptyset \subset X^0 \subset X^1 \subset \cdots \subset X^{n-2} \subset X^n.
\]
A (Goresky-MacPherson) \em perversity \em is a function $\mathbf{p}\colon\{2,\ldots,n\} \to \{0,1,\ldots\}$ such that $\mathbf{p}(2)=0$ and $\mathbf{p}(i)\leq \mathbf{p}(i+1) \leq \mathbf{p}(i)+1$ for all $i=2,\ldots,n-1$.

Simple examples are the zero-perversity $\mathbf{0}$ and the top perversity $\mathbf{t}$ with $\mathbf{t}(i)=i-2$.

For a perversity $\mathbf{p}$ the intersection homology $I^\mathbf{p}H_\ast(X)$ is the homology of a subcomplex $I^\mathbf{p}C_\ast(X)$ of the ordinary chains $C_\ast(X)$. If $X$ admits a PL-structure, a PL-chain $\xi\in C_r(X)$ is called $\mathbf{p}$-allowable, if its support $|\xi|\subset X$ satisfies
\begin{eqnarray*}
 \dim(|\xi|\cap X^{n-k}) & \leq & r-k+\mathbf{p}(k)
\end{eqnarray*}
for all $k=2,\ldots,n$. The subcomplex $I^\mathbf{p}C_\ast(X)$ then consists of those $\xi$ for which $\xi$ and $\partial \xi$ are $\mathbf{p}$-allowable. For more details see \cite{gormac, gorma2} and \cite{kirwoo}.

For normal pseudomanifolds there are canonical isomorphisms
\[
 I^\mathbf{t}H_\ast(X)\,\,\,\cong\,\,\,H_\ast(X)\hspace{0.4cm}\mbox{and}\hspace{0.4cm}I^\mathbf{0}H_\ast(X)\,\,\, \cong \,\,\, H^{n-\ast}(X),
\]
see \cite[\S 4.5]{kirwoo}. It follows from Lemma \ref{localstrat} below that $\mathcal{M}_d(\ell)$ is a normal pseudomanifold whenever it is a pseudomanifold.

One of the features of intersection homology is that it satisfies Poincar\'e duality when using field coefficients. We will also need a Lefschetz duality version for which we require pseudomanifolds with boundary. Basically, a pseudomanifold with boundary $X$ is such that $X-\partial X$ is an $n$-dimensional pseudomanifold, and $\partial X$ is an $(n-1)$-dimensional pseudomanifold which has a neighborhood in $X$ stratified homeomorphic to $\partial X\times [0,1)$, see \cite[\S 4]{fried2}. For a compact, orientable $n$-dimensional pseudomanifold with boundary we then get isomorphisms
\begin{eqnarray*}
 I^\mathbf{p}H_i(X;\F)&\cong & I^{\mathbf{t}-\mathbf{p}}H_{n-i}(X,\partial X;\F)
\end{eqnarray*}
for all $i=0,\ldots,n$, assuming that $\F$ is a field, see \cite[\S 4]{fried2}. 

\section{Linkage spaces as pseudomanifolds}
\label{section_pseudo}
We want to describe the stratification of $\mathcal{M}_d(\ell)$. This is basically given by $\mathcal{M}_k(\ell)$ where $k < d$. However, the natural map $\mathcal{M}_k(\ell)\to \mathcal{M}_d(\ell)$ is not injective. But it induces an injection $\mathcal{N}_k(\ell) \to \mathcal{M}_d(\ell)$.

\begin{definition}
 Let $x=[x_1,\ldots,x_n]\in \mathcal{M}_d(\ell)$. Then the \em rank \em of $x$ is defined as the dimension of the vector space spanned by $x_1,\ldots,x_n$.
\end{definition}

Clearly, ${\rm rank}\,(x) \leq d$, and since the $x_i$ are linearly dependent, we have ${\rm rank}\,(x) \leq n-1$. We will however be more interested in the case $n >d$, as for $n\leq d$ we get that $\mathcal{M}_d(\ell)$ is contractible or empty.

Furthermore, if $k = {\rm rank}\,(x)$, then $x$ is in the image of the natural map $\mathcal{M}_k(\ell)\to \mathcal{M}_d(\ell)$.

\begin{lemma}
Let $\ell\in \R^n$ be a length vector, and $2\leq k < d$. Then there is an inclusion $\mathcal{N}_k(\ell)\subset \mathcal{M}_d(\ell)$, and $\mathcal{N}_k(\ell)$ consists of all elements $x\in \mathcal{M}_d(\ell)$ with ${\rm rank}\,(x) \leq k$.
\end{lemma}

\begin{proof}
First note that if $x\in \mathcal{M}_d(\ell)$ has rank $k$, we can find an element $A\in \SO(d)$ such that $Ax_1,\ldots,Ax_n$ span $\R^k\times \{0\}\subset \R^d$. Hence $x$ is in the image of the natural map $\mathcal{M}_k(\ell) \to \mathcal{M}_d(\ell)$. Now assume that $x,y\in \mathcal{M}_k(\ell)$ have the same image in $\mathcal{M}_d(\ell)$. Then ${\rm rank}\,(x)={\rm rank}\,(y)=m \leq k$ for some $m$.

We can then assume that $x_1,\ldots,x_n$ span $\R^m\times\{0\}$ and also $y_1,\ldots,y_n$ span $\R^m\times \{0\}$. Now if there is an $A\in \SO(d)$ with $Ax_i=y_i$ for all $i=1,\ldots,n$, we get that $A$ keeps $\R^m$ invariant, and therefore $A|\in \OO(m)$. But this means that $x$ and $y$ represent the same element in $\mathcal{N}_k(\ell)$.
\end{proof}

In order to understand the local structure of $\mathcal{M}_d(\ell)$, choose $x\in \mathcal{M}_d(\ell)$ with ${\rm rank}\,(x)=k \leq d\leq n$, and represent this point by $(x_1,\ldots,x_n)$. Note that if $d=n$, we have $k < n$, as the coordinates of $x$ are linearly dependent.

Now rotate $x_1$ into position $e_1\in \R^d$. Let us assume that $k\geq 2$, which is always the case if $\ell$ is generic. Then there is another element not contained in $\R\times \{0\}$, and after reordering, we may assume it is $x_2$, now using a rotation from $\SO(d-1)$ (fixing the first coordinate), we can assume that $x_2\in S^1-S^0$. Repeating this, we can represent $x$ by an element $(x_1,\ldots,x_n)$ with
\begin{eqnarray*}
 x_1 & = & e_1 \\
 x_2 &\in & S^1 - S^0\\
&\vdots & \\
x_k & \in & S^{k-1} - S^{k-2}. 
\end{eqnarray*}
Since ${\rm rank}\,(x)= k$, we get that $x_{k+1},\ldots,x_n\in S^{k-1}\subset S^{d-1}$. If $k < n-1$, we can assume that the $x_{k+1},\ldots,x_n$ are not collinear: If they are, they cannot be multiples of $x_k$, since $\sum\ell_ix_i=0$ and $x_1,\ldots,x_{k-1}\in \R^{k-1}$. In that case we can just replace $x_k$ with $x_{k+1}$.

Also, if $k= n-1$, it follows that $x_k$ and $x_{k+1}$ are not collinear by the same argument.

We can therefore assume that after a permutation of coordinates we have $x_1=e_1$, $x_i \in S^{i-1}-S^{i-2}$ for $i=2,\ldots,k$, $x_{k+1},\ldots x_n \in S^{k-1}$ and $x_{n-1},x_n$ are not collinear. Furthermore, the group $\SO(d-k)$ fixes all $x_i$.

\begin{lemma}\label{localmanifold}
Let $\ell\in \R^n$ be a length vector with $n> d\geq 3$. Let $x\in \mathcal{M}_d(\ell)$ satisfy ${\rm rank}\,(x) \geq d-1$. Then $x$ has a neighborhood homeomorphic to $$\R^{(n-3)(d-1)-\frac{(d-2)(d-3)}{2}}.$$
\end{lemma}

\begin{proof}
We can use the description for $x$ given above the lemma. That is, we can represent $x$ by an element $(x_1,\ldots,x_n)$ such that
\begin{eqnarray*}
 x_1 & = & e_1 \\
 x_2 &\in & S^1 - S^0\\
&\vdots & \\
x_{d-1} & \in & S^{d-2} - S^{d-3},
\end{eqnarray*}
and $x_d,\ldots,x_n\in S^{d-1}$ with $x_{n-1}$ and $x_n$ not collinear. If ${\rm rank}\,(x)=d-1$, we actually have all $x_d,\ldots,x_n\in S^{d-2}$, otherwise we can assume that $x_d\in S^{d-1}-S^{d-2}$.

In order to describe points near $x$, we have to consider points near $(x_1,\ldots,x_n)$, so we can let them vary in small discs $D^{d-1}$. But notice that for nearby points $y$ we always get ${\rm rank}\,(y)\geq {\rm rank}\,(x)$. After a rotation, we therefore get
\begin{eqnarray*}
 y_1 & = & e_1 \\
 y_2 &\in & S^1 - S^0\\
&\vdots & \\
y_{d-1} & \in & S^{d-2} - S^{d-3}. 
\end{eqnarray*}
Furthermore, $y_d,\ldots,y_n \in S^{d-1}$, and there are no further rotations possible. The point $y_2$ can therefore freely vary in a small disc $D^1$, $y_3$ in a small disc $D^2$, etc. The points $y_d,\ldots,y_n$ can vary in $S^{d-1}$, but only up to $y_{n-2}$ we can vary them freely. The last two $y_{n-1}, y_n$ have to connect the endpoint of the linkage given by the first $n-2$ elements to the origin. Since we can assume $y_{n-1}$ and $y_n$ to be not collinear, this is possible near $x$, and there is a $(d-2)$-dimensional sphere of possibilities. The dimension of the neighborhood is therefore
\[
 1+2+\cdots +(d-2) + (d-1)(n-1-d) +(d-2) 
\]
which is easily seen to be $(n-3)(d-1)-\frac{(d-2)(d-3)}{2}$.
\end{proof}

The following lemma was proven in \cite{fotloz} in the case $d=5$.

\begin{lemma}\label{localstrat}
Let $\ell\in \R^n$ be a length vector, and $2\leq k < d-1 \leq n-1$. If ${\rm rank}\,(x)=k$, then $x$ has a neighborhood homeomorphic to
\[
 \R^{\frac{(k-1)k}{2}+(n-k-1)(k-1)-1} \times ((\R^{d-k})^{n-1-k})/\SO(d-k),
\]
where $\SO(d-k)$ acts diagonally on $(\R^{d-k})^{n-1-k}$ in the standard way. Furthermore, if $k<m \leq d-2$, the points in this neighborhood corresponding to points in $\mathcal{N}_m(\ell)$ are in
\[
 \R^{\frac{(k-1)k}{2}+(n-k-1)(k-1)-1} \times ((\R^{m-k})^{n-1-k})/\OO(m-k).
\]

\end{lemma}

\begin{proof}
The proof is similar to the proof of Lemma \ref{localmanifold}. The first $k$ points can vary in a $1+2+\cdots + k-1$-dimensional euclidean space, the next $n-2-k$ points can vary in $\R^{d-1}$, but with $\SO(d-k)$ acting diagonally on $(\R^{d-1})^{n-2-k}$. We have to think of each copy of $\R^{d-1}$ as a small disc neighborhood of a point $x_i\in S^{k-1}\times \{0\}\subset S^{d-1}$ sitting in $\R^d$, and $\SO(d-k)$ acts on the last $d-k$ variables in the usual way.

These $n-2-k$ points therefore produce a factor $(\R^{k-1})^{n-k-2}$ where there is no action, and a factor $(\R^{d-k})^{n-k-2}/\SO(d-k)$ where $\SO(d-k)$ acts diagonally, and in the usual way on each factor $\R^{d-k}$.

Finally, the points $x_{n-1}$ and $x_n$, which we can assume to be non-collinear, connect up the points $0\in \R^d$ and $y=\sum_{i=1}^{n-2}\ell_i x_i$. This gives rise to a sphere of dimension $d-2$. Since $y\in \R^k\times\{0\} \subset \R^d$, variations leading to points in $\R^k$ give rise to a sphere of dimension $S^{k-2}$, on which $\SO(d-k)$ acts trivially. This gives rise to another trivial factor $\R^{k-2}$ in the neighborhood of $x$. If we vary $x_{n-1}$ and $x_n$ within $\R^d$, we get another factor $\R^{d-k}$, on which $\SO(d-k)$ acts in the usual way. We therefore get another factor in the quotient.

To get points in $\mathcal{N}_m(\ell)$ we simply have to make sure the last $d-m$ coordinates stay $0$. We have to pass to $\OO(m-k)$, since $((\R^{m-k})^{n-1-k})/\SO(m-k)\to ((\R^{d-k})^{n-1-k})/\SO(d-k)$ is not injective.
\end{proof}

This implies that for $d\leq n-1$ the space $\mathcal{M}_d(\ell)$ carries the structure of a pseudo-manifold with stratification given by
\[
 \emptyset \subset \mathcal{N}_2(\ell) \subset \mathcal{N}_3(\ell) \subset \cdots \subset \mathcal{N}_{d-2}(\ell) \subset \mathcal{M}_d(\ell).
\]
We would like to give this pseudo-manifold a piecewise-linear structure. To see that this is possible, note that $\mathcal{C}_d(\ell)$ is a real-analytic manifold with $\SO(d-1)$ acting real-analytically. The submanifolds $\mathcal{C}_k(\ell)$ for $k < d$ are not $\SO(d-1)$-invariant, but the $\SO(d-1)$-orbits of these sets are easily seen to be subanalytic $\SO(d-1)$-invariant closed subsets of $\mathcal{C}_d(\ell)$. By \cite[Thm.B]{illman} $\mathcal{C}_d(\ell)$ can be given a $\SO(d-1)$-equivariant triangulation which gives $\mathcal{M}_d(\ell)$ a triangulation such that each $\mathcal{N}_k(\ell)$ is a subcomplex.

Let us define
\begin{eqnarray}\label{dimension_formula}
 \mathbf{d}_d^n&=& (n-3)(d-1)-\frac{(d-2)(d-3)}{2},
\end{eqnarray}
which is the dimension of $\mathcal{M}_d(\ell)$ for $n>d$ by Lemma \ref{localmanifold}. It then follows easily that the codimension of $\mathcal{N}_{d-k}(\ell)$ in $\mathcal{M}_d(\ell)$ is
\begin{eqnarray}\label{codimension_formula}
 \mathbf{c}_{d,k}^n & = & k(n-d) + \frac{k(k-1)}{2}
\end{eqnarray}
for $k=2,\ldots,d-2$.

\section{The intersection ring of a pseudomanifold}\label{section_interring}

Let $X$ be a compact, oriented $n$-dimensional pseudomanifold. In \cite{gormac,gorma2} Goresky-MacPherson define the intersection pairing
\[
 \cap\colon I^{\mathbf{p}}H_i(X)\times I^{\mathbf{q}}H_j(X) \to I^\mathbf{r}H_{i+j-n}(X)
\]
where $\mathbf{p}$, $\mathbf{q}$ and $\mathbf{r}$ are perversities such that $\mathbf{p}+\mathbf{q}\leq \mathbf{r}$, and show that it does not depend on the stratification of $X$. Furthermore, $I^{\mathbf{0}}H_n(X)$ contains a fundamental class $[X]$ which serves as a unit.

Now let $k,m>0$ and assume that $\mathbf{p}_0,\ldots,\mathbf{p}_k$ is a sequence of perversities such that for all $i,j\geq 0$ with $i+j\leq k$ we have $\mathbf{p}_i+\mathbf{p}_j\leq \mathbf{p}_{i+j}$. For $0\leq r \leq k$ define
\begin{eqnarray*}
 I_{\mathbf{p}_\cdot}H^{rm}(X)&=& I^{\mathbf{p}_r}H_{n-rm}(X).
\end{eqnarray*}
For $r+s\leq k$ the intersection pairing induces a multiplication
\[
 \cdot \colon I_{\mathbf{p}_\cdot}H^{rm}(X) \times I_{\mathbf{p}_\ast}H^{sm}(X) \to I_{\mathbf{p}_\cdot}H^{(r+s)m}(X)
\]
which turns
\begin{eqnarray*}
 I_{\mathbf{p}_\cdot}H^{\ast m}(X)&=&\bigoplus_{r=0}^k I_{\mathbf{p}_\ast}H^{rm}(X)
\end{eqnarray*}
into a graded ring with unit. We call this ring the \em intersection ring of $X$ with respect to $\mathbf{p}$ and $m$\em. If $\mathbf{p}_r$ is the $\mathbf{0}$-perversity for all $r\geq 0$, we may choose $k=\infty$.

The subring generated by the elements of $I_{\mathbf{p}_\ast}H^m(X)$ is also a graded ring with unit, and we call it the \em reduced intersection ring of $X$ with respect to $\mathbf{p}_\ast$ and $m$\em. We denote it by
\[
 I_{\mathbf{p}_\cdot}\tilde{H}^{\ast m}(X).
\]

If a pseudomanifold admits a stratification whose strata have only certain codimensions, a perversity only has to consider these codimensions. The relevant perversities for intersection homology of $\mathcal{M}_d(\ell)$ are therefore given by non-decreasing sequences of integers $(0,p_2,p_3,\ldots,p_{d-2})$ for which we have
\begin{eqnarray*}
 p_2 & \leq & 2(n-d)-1\\
 p_{i+1}-p_i & \leq & n-d+i
\end{eqnarray*}
for all $i=2,\ldots,d-3$. The top perversity is thus given by
\begin{eqnarray*}
 \mathbf{t}_n&=&(0,\mathbf{c}^n_{d,2}-2,\ldots,\mathbf{c}^n_{d,d-2}-2).
\end{eqnarray*}

We know from \cite{schuet} that $H_{\mathbf{d}_d^n}(\mathcal{M}_d(\ell))\cong \Z$, so there is a fundamental class that we can write as $[\mathcal{M}_d(\ell)]$.

Also, if $J\subset \{1,\ldots,n-1\}$ we can define a new length vector $\ell_J\in \R^{n-|J|}$ given by
\begin{eqnarray*}
 \ell_J&=&(\ell_{i_1},\ldots,\ell_{i_{n-1-|J|}},\ell_n+\ell_{j_1}+\cdots + \ell_{j_{|J|}}),
\end{eqnarray*}
where $J=\{j_1,\ldots,j_{|J|}\}$ and $\{i_1,\ldots,i_{n-1-|J|}\}$ denotes the complement of $J$ in $\{1,\ldots,n-1\}$. 

We then get a natural inclusion $\mathcal{M}_d(\ell_J) \hookrightarrow \mathcal{M}_d(\ell)$, and provided that $|J|\leq n-d-1$ we have a fundamental class $[\mathcal{M}_d(\ell_J)]\in H_{\mathbf{d}_d^{n-|J|}}(\mathcal{M}_d(\ell_J))$.

Note that if $J\cup\{n\}$ is long, we get $\mathcal{M}_d(\ell_J)=\emptyset$ and the fundamental class is just $0$.

\begin{definition}
 Let $0\leq k \leq n-d-1$. The perversity $\mathbf{p}_k$ is defined as
\begin{eqnarray*}
 \mathbf{p}_k &=& (0,2k,3k,\ldots,(d-2)k).
\end{eqnarray*}

\end{definition}

\begin{lemma}
 Let $\ell\in \R^n$ be a generic length vector and $J\subset \{1,\ldots,n-1\}$ satisfy $|J| \leq n-d-1$. Then $[\mathcal{M}_d(\ell_J)]$ represents a well-defined homology class
\begin{eqnarray*}
 [\mathcal{M}_d(\ell_J)]&\in& I^{\mathbf{p}_{|J|}}\!H_{\mathbf{d}_d^{n-|J|}}(\mathcal{M}_d(\ell)).
\end{eqnarray*}

\end{lemma}

\begin{proof}
 We need to show that
\begin{eqnarray*}
 \dim (\mathcal{M}_d(\ell_J) \cap \mathcal{N}_{d-k}(\ell)) & \leq & \dim \mathcal{M}_d(\ell_j) - {\rm co}\!\dim \mathcal{N}_{d-k}(\ell) + k|J|
\end{eqnarray*}
for all $k=2,\ldots,d-2$. Since $\mathcal{M}_d(\ell_J)\cap \mathcal{N}_{d-k}(\ell) = \mathcal{N}_{d-k}(\ell_J)$, this is equivalent to showing
\begin{multline*}
 (n-|J|-3)(d-k-1) - \frac{(d-k-2)(d-k-3)}{2} \,\,\,\leq \\
(n-|J|-3)(d-1) - \frac{(d-2)(d-3)}{2} - k(n-d) - \frac{k(k-1)}{2} + k|J|.
\end{multline*}
A straightforward calculation shows that this is indeed an equality.
\end{proof}

The relevant intersection ring for $\mathcal{M}_d(\ell)$ is obtained using $m=d-1$ and the perversity $\mathbf{p}_1$. To simplify notation and since we are mainly interested in the reduced intersection ring we will write
\[
I\!H^{\ast(d-1)}(\mathcal{M}_d(\ell)) 
\]
for the reduced intersection ring.

The relevant intersection homology groups are given by $I^{\mathbf{p}_k}\!H_{\mathbf{d}_d^{n-k}}(\mathcal{M}_d(\ell))$, and we have
\begin{eqnarray*}
 \mathbf{d}_d^{n-k}+\mathbf{d}_d^{n-j} - \mathbf{d}_d^n & = & \mathbf{d}_d^{n-(k+j)}
\end{eqnarray*}
and the perversities satisfy
\begin{eqnarray*}
 \mathbf{p}_j+\mathbf{p}_k& = & \mathbf{p}_{j+k}
\end{eqnarray*}
for $k,j = 0,\ldots,n-d-1$. Note that products with $j+k\geq n-d$ are considered as $0$.

The fundamental class $[\mathcal{M}_d(\ell)]\in I\!H^0(\mathcal{M}_d(\ell))$ is the unit of both the intersection ring and the reduced intersection ring, which follows immediately from \cite[Thm.1]{gormac}.

\begin{remark}
 For $d=3$ we can use the $\mathbf{0}$-perversity and the intersection ring consists of the cohomology ring $H^\ast(\mathcal{M}_3(\ell))$ which has been calculated by Hausmann and Knutson \cite{hauknu}.
\end{remark}

\section{A Morse function for linkage spaces}\label{section_morse}

In \cite{schuet} the function $F:\mathcal{C}_d(\ell)\to \R$ given by
\begin{eqnarray*}
 F(z_1,\ldots,z_{n-1})&=&p_1(z_{n-1}),
\end{eqnarray*}
where $p_1:\R^d\to \R$ is projection to the first coordinate, was shown to be a $\SO(d-1)$-invariant Morse-Bott function whose critical manifolds consist of
\begin{itemize}
 \item $\mathcal{C}_d(\ell^-)$ as the absolute minimum, $\ell^-=(\ell_1,\ldots,\ell_{n-2},\ell_n-\ell_{n-1})$.
 \item $\mathcal{C}_d(\ell^+)$ as the absolute maximum, $\ell^+=(\ell_1,\ldots,\ell_{n-2},\ell_n+\ell_{n-1})$.
 \item $S^{d-2}_J$, a $(d-2)$-dimensional sphere for every subset $J\subset \{1,\ldots,n-2\}$ for which $J\cup \{n\}$ is $\ell$-short and $J\cup \{n-1,n\}$ is $\ell$-long, whose index is $(n-3-|J|)(d-1)$.
\end{itemize}
Using induction, and noting \cite[\S 4]{schuet}, we can construct an $\SO(d-1)$-invariant Morse-Bott function $\bar{F}:\mathcal{C}_d(\ell)\to \R$ whose critical manifolds are all spheres of dimension $d-2$, which is perfect in the sense that the number of critical manifolds of index $k(d-1)$ agrees with the $(2k)$-th Betti number of the moduli space $\mathcal{M}_3(\ell)$, and there are no critical manifolds of other indices.

We can furthermore assume that different critical manifolds have different values under $\bar{F}$. Choosing a sequence of regular values $x_i$ such that the interval $(x_{i-1},x_i)$ contains exactly one critical value, we get a filtration
\[
 \emptyset \subset \mathcal{M}^0 \subset \mathcal{M}^1 \subset \cdots \subset \mathcal{M}^r = \mathcal{M}_d(\ell)
\]
by $\mathcal{M}^i=f^{-1}((-\infty,x_i])$, where $f:\mathcal{M}_d(\ell)\to \R$ is induced by $\bar{F}$.

Note that $\mathcal{M}^i$ is a pseudomanifold with boundary, and to understand the intersection homology of the pair $(\mathcal{M}^i,\mathcal{M}^{i-1})$, we need to understand the intersection homology of the normal bundle of the critical manifold $S^{d-2}$ relative to its boundary.

Let $N(S^{d-2})$ denote the normal bundle of the critical manifold $S^{d-2}\subset \mathcal{C}_d(\ell)$ and let $p\in S^{d-2}$. The normal space at $p$ is denoted by $N_p(S^{d-2})$.

\begin{lemma}
 $\SO(d-1)$ acts transitively on $S^{d-2}$, furthermore, for $p\in S^{d-2}$ the stabilizer subgroup is isomorphic to $\SO(d-2)$, and we have $N_p(S^{d-2})$ is $\SO(d-2)$-equivariantly homeomorphic to $(\R^{d-1})^{n-3}$, where $\SO(d-2)$ acts diagonally on the $n-3$ copies of $\R^{d-1}$, and the action of $\SO(d-2)$ on $\R^{d-1}$ fixes the first coordinate, and is the standard action on the remaining $d-2$ coordinates.
\end{lemma}

\begin{proof}
 We start with the critical manifold $S_J$ from $F$, for which the statement is easy to see by the explicit description in \cite{schuet}. To get the same statement for the critical manifolds of $\bar{F}$ recall that $\bar{F}$ is build by induction. The critical submanifolds $\mathcal{C}_d(\ell^{\pm})$ of $F$ have trivial normal bundle \cite[Lemma 3.3]{schuet}, which gives an extra copy of $\R^{d-1}$ on which $\SO(d-2)$ acts in the described way. The map $\bar{F}$ can then be constructed so that the statement holds.
\end{proof}

Let $k_i(d-1)$ be the index of the critical manifold $S^{d-2}$ of $\bar{F}$ whose critical value is in $(x_{i-1},x_i)$. Then denote
\begin{eqnarray*}
 \mathcal{N}_i&=& (D^{d-1})^{k_i}\times (D^{d-1})^{n-3-k_i}/\SO(d-2),
\end{eqnarray*}
where $D^{d-1}\subset \R^{d-1}$ is the usual closed unit ball. It follows that $\mathcal{M}_i$ is homeomorphic to $\mathcal{M}_{i-1}\cup \mathcal{N}_i$ and with the excision properties for pseudomanifolds with boundary we get
\begin{eqnarray*}
 I^\mathbf{p}H_\ast(\mathcal{M}_i,\mathcal{M}_{i-1})&\cong & I^\mathbf{p} H_\ast(\mathcal{N}_i,\partial_-\mathcal{N}_i),
\end{eqnarray*}
where $\partial_-\mathcal{N}_i= \partial((D^{d-1})^{k_i})\times (D^{d-1})^{n-3-k_i}/\SO(d-2)$ for any perversity.

Let us denote $\mathcal{N}_i^-=(D^{d-1})^{k_i}/\SO(d-2)$ and $\partial \mathcal{N}_i^-=\partial((D^{d-1})^{k_i})/\SO(d-2)$. The obvious inclusion $\mathcal{N}_i^-\subset \mathcal{N}_i$ induces a homotopy equivalence of pairs $(\mathcal{N}_i^-,\partial \mathcal{N}_i^-)\simeq (\mathcal{N}_i,\partial_- \mathcal{N}_i)$, but this does not induce an isomorphism on intersection homology in general. In fact, $\mathcal{N}_i^-$ is not a pseudomanifold for $k_i\leq d-3$.

However, if $\mathcal{N}_i^-$ is a pseudomanifold (with boundary), its dimension is $\mathbf{d}_d^{k_i+3}$, and
\begin{eqnarray*}
 I^{\mathbf{0}}H_\ast(\mathcal{N}_i^-,\partial \mathcal{N}_i^-) &\cong & I^{\mathbf{t}}H_{\mathbf{d}_d^{k_i+3}-\ast}(\mathcal{N}_i^-)
\end{eqnarray*}
by Lefschetz duality, and since the latter is just ordinary homology of a contractible space, we have
\begin{eqnarray*}
 I^{\mathbf{0}}H_{\mathbf{d}_d^{k_i}}(\mathcal{N}_i^-,\partial \mathcal{N}_i^-) &\cong & \Z
\end{eqnarray*}
and all other groups are trivial. Since we are interested in $\Z$ coefficients, we have to be slightly careful with Lefschetz duality and torsion. To see that no torsion occurs, we use \cite[Cor.4.4.3]{fried2}, note that since we use the top perversity, the condition of being locally $(\mathbf{p},\Z)$-torsion free is trivial. Also, \cite[Cor.4.4.3]{fried2} is stated for Poincar\'e duality, but because of the way Lefschetz duality is derived from Poincar\'e duality in \cite{fried2}, the result also holds for Lefschetz duality.

The inclusion $\mathcal{N}_i^-\subset \mathcal{N}_i$ is stratum preserving, but the codimensions of the strata are different. In particular, the inclusion is not placid in the sense of \cite[\S 4.8]{kirwoo}. In order to get a homomorphism between intersection homology groups of $(\mathcal{N}_i^-,\partial \mathcal{N}_i^-)$ and $(N_i,\partial_-\mathcal{N}_i)$ we have to vary the perversities for them. The following lemma gives a criterion for obtaining such a homomorphism; its proof is analogous to the proof of \cite[Ex.4.8.2]{kirwoo}.

\begin{lemma}\label{pseudofunct}
 Let $X$, $Y$ be pseudomanifolds, $\mathbf{p}$, $\mathbf{q}$ perversities, and $f:X\to Y$ a stratum preserving map for some stratifications of $X$ and $Y$. Then $f$ induces a map on homology
\[
 f_\ast:I^\mathbf{p}H_\ast(X)\to I^\mathbf{q}H_\ast(Y)
\]
provided that
\begin{eqnarray*}
 \mathbf{q}(k)-k & \geq & \mathbf{p}(m_k)-m_k
\end{eqnarray*}
where $m_k$ is the minimal codimension of a stratum $S$ with $f(S)$ in a stratum of codimension $k$.
\end{lemma}

Therefore inclusion induces a map on intersection homology by increasing the perversity for $(\mathcal{N}_i,\partial_-\mathcal{N}_i)$, that is, we have a well defined homomorphism
\[
 i_\ast:I^\mathbf{0}H_\ast(\mathcal{N}_i^-,\partial \mathcal{N}_i^-)\to I^{\mathbf{p}_{n-3-k_i}}H_\ast(\mathcal{N}_i,\partial_-\mathcal{N}_i).
\]

In the next section we show that this is indeed an isomorphism.

\section{Intersection homology of Morse data}\label{section_data}

We modify our notation, for non-negative integers $m,k$ with $m\geq k$ let
\begin{eqnarray*}
 \mathcal{N}^{m,k} &=& ((D^{d-1})^k\times (D^{d-1})^{m-k})/\SO(d-2),\\
 \partial_-\mathcal{N}^{m,k} &=& (\partial((D^{d-1})^k)\times (D^{d-1})^{m-k})/\SO(d-2).
\end{eqnarray*}
The stratification is given by
\[
 \emptyset \subset \mathcal{N}^{m,k}_2 \subset \cdots \subset \mathcal{N}^{m,k}_{d-2} \subset \mathcal{N}^{m,k}
\]
where $\mathcal{N}^{m,k}_c$ is the image of $((D^{c-1})^k\times (D^{c-1})^{m-k})/\SO(c-2)$ in $\mathcal{N}^{m,k}$ for all $c=2,\ldots, d-2$.

If $l$ is another non-negative integer with $l\leq m-k$ and $m-l \geq d-2$ we have an inclusion of pairs of pseudomanifolds with boundary
\begin{eqnarray*}
 (\mathcal{N}^{m-l,k},\partial_-\mathcal{N}^{m-l,k})&\subset & (\mathcal{N}^{m,k},\partial_-\mathcal{N}^{m,k})
\end{eqnarray*}
which adds zeros into the extra coordinates. Let $s\geq 0$ with $s+l\leq m-d+2$. The inclusion then induces a map on intersection homology 
\[
 i_\ast\colon I^{\mathbf{p}_s}H_\ast(\mathcal{N}^{m-l,k},\partial_-\mathcal{N}^{m-l,k}) \to I^{\mathbf{p}_{s+l}}H_\ast(\mathcal{N}^{m,k},\partial_-\mathcal{N}^{m,k})
\]
by Lemma \ref{pseudofunct}. This is an isomorphism.

\begin{lemma}\label{switch_perv}
 Let $d,m,l,k,s$ be non-negative integers with $l\leq m-k$, $m-l\geq d-2$ and $s\leq m-l-d+2$. Then
\begin{eqnarray*}
 I^{\mathbf{p}_s}H_\ast(\mathcal{N}^{m-l,k},\partial_-\mathcal{N}^{m-l,k}) & \cong &
I^{\mathbf{p}_{s+l}}H_\ast(\mathcal{N}^{m,k},\partial_-\mathcal{N}^{m,k})
\end{eqnarray*}
and the isomorphism is induced by inclusion.
\end{lemma}

The conditions on the integers are to ensure that we do get pseudomanifolds and perversities in the sense of \cite{gormac}.

\begin{proof}
 There is an obvious retraction $r:(\mathcal{N}^{m,k},\partial_-\mathcal{N}^{m,k}) \to (\mathcal{N}^{m-l,k},\partial_-\mathcal{N}^{m-l,k})$ induced by projection, but this retraction does not preserve the stratification.

Therefore let us assume that $l=1$. In that case $r^{-1}(\mathcal{N}^{m-1,k}_c)\subset \mathcal{N}^{m,k}_{c+1}$, and we define
\begin{eqnarray*}
 \mathcal{N}^{m,k}_{c,c+1} &=& r^{-1}(\mathcal{N}^{m-1,k}_c)
\end{eqnarray*}
for all $c=2,\ldots,d-2$. We then have the stratification
\[
 \emptyset \subset \mathcal{N}^{m,k}_2 \subset \mathcal{N}^{m,k}_{2,3} \subset \mathcal{N}^{m,k}_3 \subset \cdots \subset \mathcal{N}^{m,k}_{d-2}\subset \mathcal{N}^{m,k}_{d-2,d-1}\subset \mathcal{N}^{m,k}
\]
and it is easy to check that the retraction is stratum preserving using this stratification. Notice that
\begin{eqnarray*}
 \dim \mathcal{N}^{m,k}_{c,c+1}&=&\dim \mathcal{N}^{m,k}_c+1.
\end{eqnarray*}
To get an appropriate perversity $\mathbf{p}_{s+1}'$, we need entries for each $\mathcal{N}^{m,k}_{c,c+1}$ which can be at most one less than the entry for $\mathcal{N}^{m,k}_c$, for all $c=2,\ldots,d-2$. We denote these entries by $\mathbf{p}_{s+1}'(c,c+1)$ and set them to be
\begin{eqnarray*}
 \mathbf{p}_{s+1}'(c,c+1)&=& \mathbf{p}_{s+1}(c) - 1 \,\,\,=\,\,\, c(s+1) - 1.
\end{eqnarray*}
The other entries are the same:
\begin{eqnarray*}
 \mathbf{p}_{s+1}'(c)&=& \mathbf{p}_{s+1}(c)
\end{eqnarray*}
Since intersection homology does not depend on the stratification by \cite{gorma2}, we get
\begin{eqnarray*}
 I^{\mathbf{p}_{s+1}}H_\ast(\mathcal{N}^{m,k},\partial_-\mathcal{N}^{m,k}) & \cong &
I^{\mathbf{p}_{s+1}'}H_\ast(\mathcal{N}^{m,k},\partial_-\mathcal{N}^{m,k}).
\end{eqnarray*}
We need to check the conditions of Lemma \ref{pseudofunct} to get an induced map
\[
 r_\ast\colon I^{\mathbf{p}_{s+1}'}H_\ast(\mathcal{N}^{m,k},\partial_-\mathcal{N}^{m,k}) \to I^{\mathbf{p}_s}H_\ast(\mathcal{N}^{m-1,k},\partial_-\mathcal{N}^{m-1,k}).
\]
Note that the strata of minimal codimension mapping into $\mathcal{N}^{m-1,k}_c-\mathcal{N}^{m-1,k}_{c-1}$ are $\mathcal{N}^{m,k}_{c,c+1}-\mathcal{N}^{m,k}_c$. The codimension of $\mathcal{N}^{m-1,k}_c-\mathcal{N}^{m-1,k}_{c-1}$ is given by
\begin{eqnarray*}
 \mathbf{c}_{d,d-c}^{m+2} &=& (d-c)(m+2-d)+\frac{(d-c)(d-c-1)}{2}
\end{eqnarray*}
using (\ref{codimension_formula}). Therefore, the codimension of $\mathcal{N}^{m,k}_{c,c+1}-\mathcal{N}^{m,k}_c$ is $\mathbf{c}_{d,d-c}^{m+3}-1 = \mathbf{c}_{d,d-c}^{m+2}+c-1$.

It follows that
\begin{eqnarray*}
 \mathbf{p}_s(c) - \mathbf{c}_{d,d-c}^{m+2} & = & \mathbf{p}_{s+1}'(c,c+1) - (\mathbf{c}_{d,d-c}^{m+2}+c-1)
\end{eqnarray*}
which fits exactly the condition needed for Lemma \ref{pseudofunct} to apply. To see that $r_\ast$ is the inverse isomorphism for $i_\ast$, notice that the obvious strong deformation retraction between the identity and $i\circ r$ is also stratum preserving, and can therefore be used to construct the isomorphism, compare \cite[Prop.2.1]{friedm}.

This finishes the case $l=1$. For $l>1$ simply iterate this argument $l$ times.
\end{proof}

We saw in the previous section how to calculate $I^\mathbf{0}H_\ast(\mathcal{N}^{k,k},\partial_-\mathcal{N}^{k,k})$ using ordinary homology and cohomology. We will also need calculations for other perversities.

To simplify notation, write $\mathcal{N}^k=\mathcal{N}^{k,k}$ and $\partial\mathcal{N}^k=\partial_-\mathcal{N}^{k,k}$. In \cite{schuet} a relative CW-complex structure for $(\mathcal{N}^k,\partial\mathcal{N}^k)$ was given and used to calculate rational homology groups. We want to use this structure to also do intersection homology calculations. However, one has to be careful with using arbitrary CW-decompositions for intersection homology calculations, compare \cite[Appendix]{macvil}. We therefore repeat the construction of the CW-structure and justify its use for intersection homology calculations.

Recall that $\mathcal{N}^k=(D^{d-1})^k/\SO(d-2)$ where $\SO(d-2)$ acts on $D^{d-1}$ as the subgroup of $\SO(d-1)$ which fixes the first coordinate. A typical element of $\mathcal{N}^k$ is represented by $(x_1,\ldots,x_k)$ with each $x_i\in D^{d-1}$. Using an element of $\SO(d-2)$ we can assume that $x_1\in D^2\times \{0\}\subset D^{d-1}$ with the second coordinate non-negative. Let us introduce the notation
\begin{eqnarray*}
 D^c_+&=&\{(y_1,\ldots,y_c,0,\ldots,0)\in D^{d-1}\,|\, y_c\geq 0\},
\end{eqnarray*}
for $c=2,\ldots,d-2$, so that $x_1\in D^2_+$. We will also think of $D^c\subset D^{d-1}$ for all $c=1,\ldots,d-2$, occupying the first $c$ coordinates.

Using an element of $\SO(d-3)$, which is understood to fix the first two coordinates of $D^{d-1}$, we can assume that $x_2\in D^3_+$. Continuing this way we can represent the element of $\mathcal{N}^k$ by
\[
 (x_1,\ldots,x_k)\in D^2_+\times D^3_+ \times \cdots \times D^{d-2}_+ \times (D^{d-1})^{k-(d-3)}.
\]
To get a relative CW-structure for $(\mathcal{N}^k,\partial\mathcal{N}^k)$ we start with a $k$-cell $(D^1)^k$, which gives the cell structure for $\mathcal{N}^k_2$. To cover the elements of $\mathcal{N}^k_3$, we need cells of the form $(D^1)^l\times D^2_+\times (D^2)^{k-1-l}$ for $l=0,\ldots,k-1$.

Continuing, the cells needed for $\mathcal{N}^k_c$, $c=2,\ldots,d-2$, are of the form
\[
 (D^1)^{l_1}\times D^2_+\times (D^2)^{l_2}\times \cdots \times D^{c-1}_+\times (D^{c-1})^{k-(c-2)-l_1-\cdots - l_{c-2}},
\]
and to cover $\mathcal{N}^k-\mathcal{N}^k_{d-2}$ we require the cells
\[
 (D^1)^{l_1}\times D^2_+\times (D^2)^{l_2}\times \cdots \times D^{d-2}_+\times (D^{d-1})^{k-(d-3)-l_1-\cdots - l_{d-3}}.
\]
Writing $D^l$ as a $(d-2)$-dimensional column vector $\begin{pmatrix}
                                                       \ast \\ 0 
                                                      \end{pmatrix}$ with the entry $\ast$ used $(l-1)$ times, and $D^l_+$ as $\begin{pmatrix}
\ast \\ + \\0 \end{pmatrix}$ with the entry $\ast$ used $(l-2)$ times, we can write each cell as a $(d-2)\times k$ matrix with entries $0$, $\ast$ and $+$, compare \cite[\S 5]{schuet}. Note that the $k$-dimensional cell $(D^1)^k$ is represented by the zero matrix.

The boundary operator of the corresponding chain complex $C_\ast(\mathcal{N}^k,\partial\mathcal{N}^k)$ has been described explicitly in \cite{schuet}. Roughly, on the level of matrices it is obtained by summing over all matrices obtained by replacing $+$ with $0$, and coefficients either $0$ or $\pm 2$, where
$\begin{array}{cc}
 + & \ast \\ 0 & \ast  
 \end{array}$
always turns into $0$, and
$\begin{array}{cc}
  + & \ast \\ 0 & 0
 \end{array}$
turns into $\varepsilon\cdot \begin{array}{cc}0 & + \\ 0 & 0 \end{array}$ with $\varepsilon\in \{0,\pm 2\}$.

This is justified as $D^l_+\times D^l$ has as part of its boundary $D^{l-1}\times D^l$ which can be thought of as $D^{l-1}\times D^l_+\cup D^{l-1}\times D^l_-$. Using an element of $\SO(d-2)$ we can map $D^{l-1}\times D^l_-$ to $D^{l-1}\times D^l_+$, but this map affects the orientations of other discs, so that on the level of the chain complex the boundary contribution of the cell $D^{l-1}\times D^l_+$ may be $0$ or $\pm 2$.

Note that a cell $(D^1)^{l_1}\times D^2_+ \times \cdots \times (D^{d-1})^{k-L}$ intersected with $\mathcal{N}^k_c$ is another cell $(D^1)^{l_1}\times D^2_+\times \cdots \times (D^{c-1})^{k-L'}$, so the CW-structure could be considered `flag-like'. We use this to subdivide the cell structure to a flag-like triangulation without changing the chain homotopy type of the corresponding intersection homology complex.

We begin by subdividing $D^{d-1}$. For $i=1,\ldots,d-1$ let $\varepsilon_i\in \{-,0,+\}$ and
\begin{eqnarray*}
 D^{d-1}_{\varepsilon_1\cdots \varepsilon_{d-1}}&=&\{(y_1,\ldots,y_{d-1})\in D^{d-1}\,|\,y_i\sim\varepsilon_i\},
\end{eqnarray*}
where $y_i\sim 0$ means $y_i=0$, $y_i\sim +$ means $y_i\geq 0$ and $y_i\sim -$ means $y_i\leq 0$. If $2\leq c\leq d-2$, we can also write $D^{c-1}_{\varepsilon_1\cdots \varepsilon_{c-1}}=D^{d-1}_{\varepsilon_1\cdots \varepsilon_{c-1}0\cdots 0}$.

We then get a subdivision of the previous cell structure using cells of the form
\[
 D^{c_1-1}_{\varepsilon_{1\,1}\cdots \varepsilon_{1\,c_1-1}} \times \cdots \times D^{c_k-1}_{\varepsilon_{k\,1}\cdots \varepsilon_{k\,c_k-1}}
\]
subject to the condition that $c_1\geq 2$, with $\varepsilon_{1\,c_1-1}=+$ if $c_1>2$, $c_i-c_{i-1}\in \{0,1\}$ and if $c_i>c_{i-1}$, then $\varepsilon_{i\,c_i-1}=+$.

Let us denote by $C_\ast'(\mathcal{N}^k,\partial\mathcal{N}^k)$ the cellular chain complex of this subdivision. We can form the subcomplexes $I^\mathbf{p}C_\ast(\mathcal{N}^k,\partial\mathcal{N}^k)$ and $I^\mathbf{p}C_\ast'(\mathcal{N}^k,\partial\mathcal{N}^k)$ using the standard definition of $\mathbf{p}$-allowable chains. Then subdivision induces a chain map
\[
 i\colon I^\mathbf{p}C_\ast(\mathcal{N}^k,\partial\mathcal{N}^k) \to I^\mathbf{p}C_\ast'(\mathcal{N}^k,\partial\mathcal{N}^k).
\]
To obtain a chain map $p\colon I^\mathbf{p}C_\ast'(\mathcal{N}^k,\partial\mathcal{N}^k) \to I^\mathbf{p}C_\ast(\mathcal{N}^k,\partial\mathcal{N}^k)$ note that each cell in the subdivision complex satisfies
$D^{c-1}_{\varepsilon_1\cdots \varepsilon_{c-1}} \subset D^{c-1}$, and if $\varepsilon_{c-1}=+$, $D^{c-1}_{\varepsilon_1\cdots \varepsilon_{c-1}} \subset D^{c-1}_+$.

We can find a cellular homotopy $H\colon D^{c-1}\times I \to D^{c-1}$ between the identity and a map $H_1$ which is cellular when viewed as a map from the subdivision of $D^{c-1}$ to the original cell. Furthermore, we obtain another homotopy $H_+\colon D^{c-1}_+\times I\to D^{c-1}_+$ which does the same for $D^{c-1}_+$.

This can be done using induction on $c$, basically by sliding the cell $D^{c-1}_{+\cdots +}$ over $D^{c-1}$ (or $D^{c-1}_+$) while all other $c-1$-dimensional cells map into lower skeleta. These homotopies are stratum preserving, so can be used to define a chain map $p\colon I^\mathbf{p}C_\ast'(\mathcal{N}^k,\partial\mathcal{N}^k) \to I^\mathbf{p}C_\ast(\mathcal{N}^k,\partial\mathcal{N}^k)$ with $p\circ i={\rm id}_{I^\mathbf{p}C_\ast}$ and $i\circ p$ chain homotopic to ${\rm id}_{I^\mathbf{p}C_\ast'}$.

Note that the construction of the homotopies to give the chain homotopy equivalences is similar to the proof in \cite[Appendix]{macvil}.

The subdivision cell complex is regular in the sense that the attaching maps are homeomorphisms onto their image, and it has the same flag-like property as the original cell complex. If we subdivide this subdivision further by a flag-like triangulation, we can use the flag-like property to see that $I^\mathbf{p}C_\ast'(\mathcal{N}^k,\partial\mathcal{N}^k)$ has the correct chain homotopy type for intersection homology by a similar argument as above, compare also \cite[Appendix]{macvil}.

It remains to calculate the homology of $I^\mathbf{p}C_\ast(\mathcal{N}^k,\partial\mathcal{N}^k)$.

Each cell $\sigma$ is represented by a symbolic $(d-2)\times k$ matrix, whose non-zero rows (except the last one) are of the form $(0\,\cdots\,0\,+\,\ast\,\cdots \,\ast)$, with the last one of the form $(0\,\cdots\,0\,\ast\,\cdots \,\ast)$. More precisely, if we denote the number of non-zero entries in the $i$-th row, $i=1,\ldots,d-2$, by $k_i$, we have $0\leq k_1\leq k$, $0\leq k_i\leq \max\{0,k_{i-1}-1\}$ for $i=2,\ldots,d-3$, and $k_{d-2}=\max\{0,k_{d-3}-1\}$.

The dimension of the cell $\sigma$ is then given by
\begin{eqnarray*}
 \dim \sigma&=&k+k_1+\cdots +k_{d-2}.
\end{eqnarray*}
Also note that $\sigma$ is a cell in $\mathcal{N}^k_c$ if and only if $k_i=0$ for $i\geq c-1$, for all $c=2,\ldots,d-2$. Furthermore,
\begin{eqnarray*}
 \dim (\sigma\cap \mathcal{N}^k_c)&=& k + k_1 + \cdots k_{c-2}.
\end{eqnarray*}
If $\mathbf{p}$ is a perversity, the condition for the cell $\sigma$ to be $\mathbf{p}$-allowable is then simply given by
\begin{eqnarray*}
 \mathbf{c}_{d,d-c}^{k+3}-\mathbf{p}(d-c) & \leq & k_{c-1}+\cdots + k_{d-2}.
\end{eqnarray*}
Recall that $\mathbf{c}_{d,d-c}^{k+3}k$ is the codimension of $\mathcal{N}^k_c$ and $\mathbf{p}(d-c)$ is the entry corresponding to this stratum.

If we look at the perversity $\mathbf{p}_s$ with $s\in \{0,\ldots,k+2-d\}$, we can get a simpler criterion for allowability.

\begin{lemma}\label{allowable}
 Let $s\in \{0,\ldots,k+2-d\}$. For the cell $\sigma$ to be $\mathbf{p}_s$-allowable, we need
\[
 k-i+1-s \,\,\, \leq \,\,\, k_i \,\,\, \leq \,\,\, k-i+1
\]
for all $i=1,\ldots,d-3$.
\end{lemma}

\begin{proof}
 Note that $k_i\leq k-i+1$ is satisfied anyway. So assume that $k-i+1-s > k_i$ for some $i\in \{1,\ldots,d-3\}$. Then
\begin{eqnarray*}
 k-j+1-s & > & k_j
\end{eqnarray*}
for all $j\geq i$. Write $c=d-i+1$. Then
\begin{eqnarray*}
 k_{d-c-1}+k_{d-c}+\cdots +k_{d-2} & < & ck - (d-c-1)-\cdots -(d-2)+c-cs\\
&=& ck-cd+2+\cdots +(c-1)+c+(c+1)+c-cs\\
&=& c(k+3-d)+1+2+\cdots +(c-1) -cs\\
&=& \mathbf{c}_{d,c}^{k+3} -\mathbf{p}_s(c)
\end{eqnarray*}
which would contradict $\mathbf{p}_s$-allowability. The same calculation also shows that the inequality is sufficient for allowability.
\end{proof}

For $s=0$ this means that only the top-dimensional cell is allowable. As its boundary is zero, this confirms our previous calculation of $I^\mathbf{0}H_\ast(\mathcal{N}^k,\partial\mathcal{N}^k)$.

The remaining case of interest for us is when $s=1$. Recall that we assume $k\geq d-2$, so that $\mathcal{N}^k$ is a pseudomanifold with boundary. In order to get that $\mathbf{p}_1$ is a perversity, we actually need $k\geq d-1$. The dimension of $\mathcal{N}^k$ is given by $\mathbf{d}^{k+3}_d$, compare (\ref{dimension_formula}).

\begin{proposition}\label{morse_data}
 Let $k\geq d-1$.

If $d$ is odd, then
\begin{eqnarray*}
 I^{\mathbf{p}_1}H_{\mathbf{d}^{k+3}_d-r}(\mathcal{N}^k,\partial\mathcal{N}^k)&=&\left\{
\begin{array}{cl}
 \Z & r=0 \\
 \Z/2 & r=2l+1,\,l=1,\ldots,(d-3)/2\\
 0 & {\rm otherwise}
\end{array}
\right.
\end{eqnarray*}
If $d$ is even, then
\begin{eqnarray*}
 I^{\mathbf{p}_1}H_{\mathbf{d}^{k+3}_d-r}(\mathcal{N}^k,\partial\mathcal{N}^k)&=&\left\{
\begin{array}{cl}
 \Z & r=0 \\
 \Z/2 & r=2l,\,l=1,\ldots,(d-4)/2\\
 0 & {\rm otherwise}
\end{array}
\right.
\end{eqnarray*}

\end{proposition}

\begin{proof}
 We claim that $I^{\mathbf{p}_1}C_\ast(\mathcal{N}^k,\partial\mathcal{N}^k)$ is generated by one cell each in dimensions $\mathbf{d}^{k+3}_d-r$ for $r=0,2,3,\ldots,d-2$.

To see this note that for a cell to be $\mathbf{p}_1$-allowable, Lemma \ref{allowable} gives the existence of $i_0\in \{1,\ldots,d-3,d-1\}$ such that
\begin{eqnarray*}
 k_i & = & \left\{
\begin{array}{cl}
 k-i+1 & i<i_0 \\
 k-i & i\geq i_0
\end{array}
\right.
\end{eqnarray*}
The case $i_0=d-2$ is omitted, as $k_{d-2}=k_{d-3}-1$, and $i_0=d-1$ corresponds to the top-dimensional cell. The matrices of the non-top-dimensional, $\mathbf{p}_1$-allowable cells each have one occurance of a submatrix $\begin{pmatrix} + & \ast & \ast \\ 0 & 0 & + \end{pmatrix}$, while all other occurences of $+$ are in a submatrix of the form $\begin{pmatrix} + & \ast \\ 0 & + \end{pmatrix}$. Therefore the boundary of such a cell in $C_\ast(\mathcal{N}^k,\partial\mathcal{N}^k)$ can involve at most one other cell, which is also $\mathbf{p}_1$-allowable. The dimension of such a cell, depending on $i_0$, is easily seen to be $\mathbf{d}^{k+3}_d-(d-1-i_0)$ which proves the claim.

Let us denote such a cell by $\sigma_{i_0}$ with $i_0\in \{1,\ldots,d-3,d-1\}$. Clearly $\partial(\sigma_{d-1})=0$ which shows that the top-dimensional homology is $\Z$. But for $i_0\leq d-3$, we get $\partial(\sigma_{i_0})= \varepsilon_{i_0} \sigma_{i_0-1}$ with $\varepsilon_{i_0}\in \{0,\pm 2\}$. This boundary has been analyzed in detail in \cite[\S 8]{schuet}, and it is shown there that $\varepsilon$ depends on the difference of the number of non-zero entries in the $i_0$-th row with the number of non-zero-entries in the last row. More precisely, if the difference of these numbers, taken from the matrix of $\sigma_{i_0-1}$, is $l$, the coefficient is $1+(-1)^{l}$.

In particular, we have $\partial(\sigma_{d-3})=2\sigma_{d-4}$, interpreting $\sigma_0=0$, in case $d=4$. As the coefficients are alternating between $0$ and $2$ from then on, the chain complex $I^{\mathbf{p}_1}C_\ast(\mathcal{N}^k,\partial\mathcal{N}^k)$ is of the form
\[
 \Z \longrightarrow 0 \longrightarrow \Z \stackrel{\cdot 2}{\longrightarrow} \Z \stackrel{\cdot 0}{\longrightarrow} \Z \longrightarrow \cdots \stackrel{\cdot \varepsilon_1}{\longrightarrow} \Z
\]
As $\varepsilon_1=2$ for odd $d$ and $\varepsilon_1=0$ for even $d$, the result follows.
\end{proof}

\section{Generators for the reduced intersection ring}

The reduced intersection ring $I\!H^{(d-1)\ast}(\mathcal{M}_d(\ell))$ is generated by $I^{\mathbf{p}_1}H_{\mathbf{d}^{n-1}_d}(\mathcal{M}_d(\ell))$, so we begin by calculating this group.

In Section \ref{section_morse} we obtained a $\SO(d-1)$-invariant Morse-Bott function $\bar{F}:\mathcal{C}_d(\ell)\to \R$ with all critical manifolds spheres of dimension $d-2$ and of index $k(d-1)$ for some $k\in \{0,\ldots,n-3\}$. Furthermore, the number of such spheres with index $k(d-1)$ is equal to the $(2k)$-th Betti number of $\mathcal{M}_3(\ell)$. The Betti numbers of $\mathcal{M}_3(\ell)$ are well known, see \cite{hauknu, klyach}. In particular, for $\mathcal{M}_3(\ell)\not=\emptyset$ we have exactly one absolute maximum, and the number of critical manifolds of index $(n-4)(d-1)$ is equal to $1+a_1(\ell)$.

The Morse function $\bar{F}$ induces the filtration
\[
 \emptyset \subset \mathcal{M}^0 \subset \mathcal{M}^1 \subset \cdots \subset \mathcal{M}^m = \mathcal{M}_d(\ell)
\]
as in Section \ref{section_morse}.

Recall that we need $n\geq d+2$ for $\mathbf{p}_1$ to be a perversity. If $n=d+1$, $\mathcal{M}_d(\ell)$ is either empty or a sphere of dimension $\mathbf{d}^{d+1}_d$, so we do not need to consider this case.

\begin{lemma}\label{almost_calc}
 Let $d\geq 4$, $n\geq d+2$ and $\ell\in \R^n$ a generic length vector such that $\mathcal{M}_d(\ell)\not=\emptyset$. Then
\begin{eqnarray*}
 I^{\mathbf{p}_1}H_{\mathbf{d}^{n-1}_d}(\mathcal{M}^{m-1}) & \cong & \Z^{1+a_1(\ell)},\\
 I^{\mathbf{p}_1}H_{\mathbf{d}^{n-1}_d-1}(\mathcal{M}^{m-1}) & \cong & 0.
\end{eqnarray*}

\end{lemma}

\begin{proof}
 For $l\leq m-1$ we have the long exact sequence
\begin{multline*}
 \cdots \longrightarrow I^{\mathbf{p}_1}H_{r+1}(\mathcal{M}^l,\mathcal{M}^{l-1}) \longrightarrow I^{\mathbf{p}_1}H_r(\mathcal{M}^{l-1})\longrightarrow I^{\mathbf{p}_1}H_r(\mathcal{M}^l)\longrightarrow \\ I^{\mathbf{p}_1}H_r(\mathcal{M}^l,\mathcal{M}^{l-1}) \longrightarrow \cdots 
\end{multline*}
and $I^{\mathbf{p}_1}H_{r+1}(\mathcal{M}^l,\mathcal{M}^{l-1})\cong I^{\mathbf{p}_1}H_r(\mathcal{N}^{n-3,k_l},\partial_-\mathcal{N}^{n-3,k_l})$ using the notation of Section \ref{section_data}. Here $k_l$ refers to the index of the critical point contained in $\mathcal{M}^l-\mathcal{M}^{l-1}$. Then $k_l\leq n-4$ as there is only one critical point of index $(n-3)(d-1)$, and it is contained in $\mathcal{M}^m-\mathcal{M}^{m-1}$.

We need to show that
\begin{eqnarray*}
 I^{\mathbf{p}_1}H_r(\mathcal{N}^{n-3,k_l},\partial_-\mathcal{N}^{n-3,k_l}) & = & 0
\end{eqnarray*}
for $r=\mathbf{d}_d^{n-1}-1,\mathbf{d}_d^{n-1}, \mathbf{d}_d^{n-1}+1$ if $k_l<n-4$, and
\begin{eqnarray*}
 I^{\mathbf{p}_1}H_r(\mathcal{N}^{n-3,n-4},\partial_-\mathcal{N}^{n-3,n-4}) & = & \left\{
\begin{array}{cl}
 \Z & r=\mathbf{d}^{n-1} \\
 0 & r=\mathbf{d}^{n-1}\pm 1
\end{array}
\right.
\end{eqnarray*}
The latter follows directly from Lemma \ref{switch_perv}. For the former, we use Lemma \ref{switch_perv} and Lefschetz duality for pseudomanifolds with boundary to get
\begin{eqnarray*}
 I^{\mathbf{p}_1}H_r(\mathcal{N}^{n-3,k_l},\partial_-\mathcal{N}^{n-3,k_l}) & \cong & I^\mathbf{0}H_r(\mathcal{N}^{n-4,k_l},\partial_-\mathcal{N}^{n-4,k_l})\\
&\cong & I^\mathbf{t}H_{\mathbf{d}^{n-1}_d-r}(\mathcal{N}^{n-4,k_l},\partial_+\mathcal{N}^{n-4,k_l})
\end{eqnarray*}
This is the ordinary homology group $H_{\mathbf{d}^{n-1}_d-r}(\mathcal{N}^{n-4,n-4-k_l},\partial_-\mathcal{N}^{n-4,n-4-k_l})$. The homology of this pair has been calculated in \cite{schuet}. In particular, as $k_l<n-4$, $\partial_-\mathcal{N}^{n-4,n-4-k_l}\not=\emptyset$, so the homology vanishes in degrees $0,\pm 1$. Again we do not get any torsion from Lefschetz duality, as \cite[Cor.4.4.3]{fried2} applies.

Therefore no homology occurs in degree $\mathbf{d}^{n-1}_d-1$, and the rank of $I^{\mathbf{p}_1}H_{\mathbf{d}^{n-1}_d}(\mathcal{M}^l)$ increases by one exactly when there is a critical point of index $(n-4)(d-1)$ in $\mathcal{M}^l-\mathcal{M}^{l-1}$. As there are exactly $1+a_1(\ell)$ such critical points, the result follows.
\end{proof}

\begin{corollary}\label{odd_generators}
 Let $d\geq 5$ be odd, $n\geq d+2$ and $\ell\in \R^n$ a generic length vector with $\mathcal{M}_d(\ell)\not=\emptyset$. Then
\begin{eqnarray*}
 I^{\mathbf{p}_1}H_{\mathbf{d}^{n-1}_d}(\mathcal{M}_d(\ell))&\cong & \Z^{1+a_1(\ell)}.
\end{eqnarray*}

\end{corollary}

\begin{proof}
 We have the long exact sequence
\begin{multline*}
 \cdots \longrightarrow I^{\mathbf{p}_1}H_{\mathbf{d}^{n-1}_d+1}(\mathcal{M}_d(\ell),\mathcal{M}^{m-1}) \longrightarrow I^{\mathbf{p}_1}H_{\mathbf{d}^{n-1}_d}(\mathcal{M}^{m-1})\longrightarrow \\ I^{\mathbf{p}_1}H_{\mathbf{d}^{n-1}_d}(\mathcal{M}_d(\ell))\longrightarrow I^{\mathbf{p}_1}H_{\mathbf{d}^{n-1}_d}(\mathcal{M}_d(\ell),\mathcal{M}^{m-1}) \longrightarrow \cdots 
\end{multline*}
Combining Proposition \ref{morse_data} with Lemma \ref{almost_calc}, this sequence reduces to
\[
 \Z/2 \longrightarrow \Z^{1+a_1(\ell)} \longrightarrow I^{\mathbf{p}_1}H_{\mathbf{d}^{n-1}_d}(\mathcal{M}_d(\ell)) \longrightarrow 0
\]
As $\Z/2$ necessarily has to map to $0$, the result follows.
\end{proof}
It can easily be shown that $I^{\mathbf{p}_1}H_{\mathbf{d}^{n-1}_d+1}(\mathcal{M}^{m-1})=0$, so we also get
\begin{eqnarray*}
 I^{\mathbf{p}_1}H_{\mathbf{d}^{n-1}_d+1}(\mathcal{M}_d(\ell))&\cong & \Z/2
\end{eqnarray*}
for odd $d$. However, we have no real use for this result.

For even $d$ the long exact sequence turns into
\[
 \Z \longrightarrow \Z^{1+a_1(\ell)} \longrightarrow I^{\mathbf{p}_1}H_{\mathbf{d}^{n-1}_d}(\mathcal{M}_d(\ell)) \longrightarrow 0
\]
and we will see that the first map is non-trivial. This will actually simplify our calculation of the reduced intersection ring, at least after using rational coefficients. For $d$ odd we can only handle certain special cases. A particular case is given in the following example.

\begin{example}\label{shape_space}
 Let $\ell^n=(1,\ldots,1,n-2)\in \R^n$. This is the unique generic and ordered length vector with $a_1(\ell)=0$. The space $\mathcal{M}_d(\ell^n)$ is also known as the shape space $\Sigma^{n-1}_{d-1}$ of \cite{kebacl}, see \cite[Prop.A.1]{schuet}.

Let $d$ be odd and $n=d+k$ with $k\geq 1$. For $s\leq k-1$ we claim that
\begin{eqnarray*}
 I^{\mathbf{p}_s}H_{\mathbf{d}^{n-s}_d}(\mathcal{M}_d(\ell^n))&\cong & \Z.
\end{eqnarray*}
Furthermore, the reduced intersection ring agrees with the unreduced intersection ring, and we have
\begin{eqnarray*}
 I\!H^{(d-1)\ast}(\mathcal{M}_d(\ell^{d+k})) & \cong & \Z[X]/X^k
\end{eqnarray*}
with the degree of $X$ equal to $d-1$.

The proof is by induction on $k$. We will also show that $I^{\mathbf{p}_s}H_{\mathbf{d}^{n-s}_d}(\mathcal{M}_d(\ell^n))$ is generated by the fundamental class 
\[
[\mathcal{M}_d(\ell^{n-s})]\,\,\,\in\,\,\, I^{\mathbf{p}_s}H_{\mathbf{d}^{n-s}_d}(\mathcal{M}_d(\ell^n))
\]
and that
\begin{eqnarray*}
 I^{\mathbf{p}_s}H_{\mathbf{d}^{n-s}_d-1}(\mathcal{M}_d(\ell^n))&=&0.
\end{eqnarray*}
For $s=0$ and all $k$ the fundamental class part is a standard result for pseudomanifolds, and the latter part follows from Section \ref{section_data}.

For $k\geq 2$, note that $(\ell^{d+k})^-=\ell^{d+k-1}$ and $(\ell^{d+k})^+$ satisfies $\mathcal{M}_d((\ell^{d+k})^+)=\emptyset$. The $\SO(d-1)$-equivariant Morse function $F:\mathcal{C}_d(\ell^{d+k})\to \R$ from Section \ref{section_morse} has two critical manifolds, the absolute minimum $\mathcal{C}_d(\ell^{d+k-1})$ and one absolute maximum $S^{d-2}$. We thus get a filtration of $\mathcal{M}_d(\ell^{d+k})$ of the form
\[
 \emptyset \subset \mathcal{M}^{d+k} \subset \mathcal{M}_d(\ell^{d+k})
\]
where $\mathcal{M}^{d+k}$ has the homotopy type of $\mathcal{M}_d(\ell^{d+k-1})$. In fact, using the same technique as in the proof of Lemma \ref{switch_perv}, we see that
\begin{eqnarray*}
 I^{\mathbf{p}_s}H_\ast(\mathcal{M}_d(\ell^{d+k-1})) & \cong & I^{\mathbf{p}_{s+1}}H_\ast(\mathcal{M}^{d+k}).
\end{eqnarray*}
In particular,
\[
 [\mathcal{M}_d(\ell^{d+k-s-1})]\,\,\,\in\,\,\, I^{\mathbf{p}_{s+1}}H_{\mathbf{d}^{d+k-s-1}_d}(\mathcal{M}^{d+k})
\]
is a generator and
\begin{eqnarray*}
 I^{\mathbf{p}_{s+1}}H_{\mathbf{d}^{d+k-s-1}_d-1}(\mathcal{M}^{d+k})&=&0.
\end{eqnarray*}
We have the long exact sequence
\begin{multline*}
 \cdots \longrightarrow I^{\mathbf{p}_{s+1}}H_{r+1}(\mathcal{M}_d(\ell^{d+k}),\mathcal{M}^{d+k}) \longrightarrow I^{\mathbf{p}_{s+1}}H_r(\mathcal{M}^{d+k}) \longrightarrow\\ I^{\mathbf{p}_{s+1}}H_r(\mathcal{M}_d(\ell^{d+k})) \longrightarrow I^{\mathbf{p}_{s+1}}H_r(\mathcal{M}_d(\ell^{d+k}),\mathcal{M}^{d+k})
\end{multline*}
Note that
\begin{eqnarray*}
 I^{\mathbf{p}_{s+1}}H_r(\mathcal{M}_d(\ell^{d+k}),\mathcal{M}^{d+k}) &\cong & I^{\mathbf{p}_{s+1}}H_r(\mathcal{N}^{d+k-3},\partial\mathcal{N}^{d+k-3})
\end{eqnarray*}
in the notation of Section \ref{section_data}. By Lemma \ref{allowable}, the minimal dimensional cell which is $\mathbf{p}_{s+1}$-allowable is obtained by choosing each $k_i=d+k-3-i-s$ for $i=1,\ldots,d-2$. The dimension of this cell is

\[
\begin{split}
  d+k-3+(d+k-3-1-s)+\cdots +(d+k-3-(d-2)-s) & = \\
  (d+k-3-s-1)(d-1)+(s+1)+1+\cdots +(d-3) & = \\
 \mathbf{d}^{d+k-s-1}_d+(s+1)
\end{split}
\]
In particular, for $s>0$ and $r\leq \mathbf{d}^{d+k-s-1}_d$ inclusion induces an isomorphism
\begin{eqnarray*}
 I^{\mathbf{p}_{s+1}}H_r(\mathcal{M}^{d+k})&\cong& I^{\mathbf{p}_{s+1}}H_r(\mathcal{M}_d(\ell^{d+k}))
\end{eqnarray*}
and this even holds for $s=0$ as in Corollary \ref{odd_generators}. This finishes our induction step. It remains to calculate the intersection ring. Of course, $X$ corresponds to the fundamental class
\[
 [\mathcal{M}_d(\ell^{d+k-1})]\,\,\,\in\,\,\, I^{\mathbf{p}_1}H_{\mathbf{d}^{d+k-1}_d}(\mathcal{M}_d(\ell^{d+k})).
\]
Note that we think of $\mathcal{M}_d(\ell^{d+k-1})\subset \mathcal{M}_d(\ell^{d+k})$ as those points $[x_1,\ldots,x_n]$ with $x_{n-1}=-x_n$. But we could also fix a different coordinate to point in the opposite direction of the last entry: let
\begin{eqnarray*}
 \mathcal{M}_d(\ell^{d+k-1})'&=&\{ [x_1,\ldots,x_n]\in \mathcal{M}_d(\ell^{d+k})\,|\,x_{n-2}=-x_n\}.
\end{eqnarray*}
Clearly this is homeomorphic to $\mathcal{M}_d(\ell^{d+k-1})$ by permuting coordinates. From Lemma \ref{fixwhatever} below it follows that
\begin{eqnarray*}
 [\mathcal{M}_d(\ell^{d+k-1})]&=&[\mathcal{M}_d(\ell^{d+k-1})'].
\end{eqnarray*}
Therefore, as $\mathcal{M}_d(\ell^{d+k-1})$ and $\mathcal{M}_d(\ell^{d+k-1})'$ are transverse in the sense of \cite{gormac},
\begin{eqnarray*}
 X^2&=& [\mathcal{M}_d(\ell^{d+k-1})] \cdot [\mathcal{M}_d(\ell^{d+k-1})]\\
&=& [\mathcal{M}_d(\ell^{d+k-1})] \cdot [\mathcal{M}_d(\ell^{d+k-1})']\\
&=& [\mathcal{M}_d(\ell^{d+k-1})\cap \mathcal{M}_d(\ell^{d+k-1})']\\
&=& [\mathcal{M}_d(\ell^{d+k-2})]
\end{eqnarray*}
This means that $X^2$ is a generator of $I^{\mathbf{p}_2}H_{\mathbf{d}^{d+k-2}_d}(\mathcal{M}_d(\ell^{d+k}))$, and we can iterate this argument until $X^k=0$.
\end{example}

\begin{lemma}
 \label{fixwhatever}
Let $d\geq 3$, $n\geq d+2$ and $\ell^n=(1,\ldots,1,n-2)\in \R^n$. Denote
\begin{eqnarray*}
 \mathcal{M}&=&\{ [x_1,\ldots,x_n] \in \mathcal{M}_d(\ell^n)\,|\, x_{n-1}=-x_n\}\\
 \mathcal{M}'&=&\{[x_1,\ldots,x_n] \in \mathcal{M}_d(\ell^n)\,|\, x_{n-2}=-x_n\}.
\end{eqnarray*}
Then $[\mathcal{M}]=[\mathcal{M'}]\in I^{\mathbf{p}_1}H_{\mathbf{d}^{n-1}_d}(\mathcal{M}_d(\ell^n))$.
\end{lemma}

\begin{proof}
 The idea is the following: if $[x_1,\ldots,x_n]\in \mathcal{M}$, then $x_{n-2}$ is linearly independent of $x_{n-1}$, unless $[x_1,\ldots,x_n]\in \mathcal{M}\cap \mathcal{M}'$. Note that each $x_i$ for $i<n$ has to be close to $-x_n$ by the particular form of the length vector. We can therefore flip the position of $x_{n-2}$ and $x_{n-1}$ through a 1-dimensional parameter. This will define a homotopy $f:\mathcal{M}\times [0,1]\to \mathcal{M}_d(\ell^n)$ relative to $\mathcal{M}\cap\mathcal{M}'$ between the inclusion of $\mathcal{M}$ and the inclusion of $\mathcal{M}'$.

More precisely, let $x=[x_1,\ldots,x_n]\in \mathcal{M}$. Assume that $x_n=e_1$, $x_{n-1}=-e_1$ and $x_{n-2}\in S^1_-=\{(y_1,y_2,0,\ldots,0)\in S^{d-1}\,|\,y_1,y_2\leq 0\}$. Let $l_x=\|x_1+\cdots+x_{n-3}\|>0$ and $\ell_x=(l_x,1,1,n-2)$. It is easy to see that $\mathcal{C}_2(\ell_x)$ is a point for $x\in \mathcal{M}\cap \mathcal{M}'$ and a circle otherwise.

We think of $\mathcal{C}_2(\ell_x)$ as a subspace of $\mathcal{C}_d(\ell^n)$, where the first link corresponds to a rotation of $x_1+\cdots + x_{n-3}$ in the plane, and the second and third link correspond to $x_{n-2}$ and $x_{n-1}$. Indeed, denote $z_x=(x_1+\cdots+x_{n-3})/l_x\in S^1$. Then define $i_x\colon \mathcal{C}_2(\ell_x)\to \mathcal{C}_d(\ell^n)$ by
\begin{eqnarray*}
 i_x(a,b,c)&=&((az_x^{-1})x_1,\ldots,(az_x^{-1})x_{n-3},b,c)
\end{eqnarray*}
where we think of $S^1$ acting on $S^{d-1}$ by rotation of the first two coordinates, that is, as $\SO(2)$. It is easy to see that this is indeed an inclusion.

There is a unique point $(a,b,c)\in \mathcal{C}_2(\ell_x)$ such that $c=-e_1$, $b\in S^1_-$. Write $b=\exp((\pi+u_x)i)$ with $u_x\geq 0$. As $\mathcal{C}_2(\ell_x)$ is either a circle or a point, there is a unique map $h_x:[0,1]\to \mathcal{C}_2(\ell_x)$ with $p_3(h_x(t))=\exp((\pi+u_xt) i)$, where $p_3:\mathcal{C}_2(\ell_x)\to S^1$ is projection to the third coordinate.

It is now straightforward to check that $f\colon\mathcal{M}\times [0,1]\to \mathcal{M}_d(\ell^n)$ given by
\begin{eqnarray*}
 f(x,t)&=&[i_x(h_x(t))]
\end{eqnarray*}
is a well defined map which satisfies $f(x,0)$ inclusion, and $f(x,1)=q(x)$, where $q:\mathcal{M}\to \mathcal{M}_d(\ell^n)$ is inclusion followed by flipping the $(n-2)$-nd and $(n-1)$-st coordinates. As $f$ is stratum preserving, $f_0$ and $f_1$ induce the same map on intersection homology
\[
 f_{0\ast}=f_{1\ast}\colon I^\mathbf{0}H_\ast(\mathcal{M})\to I^{\mathbf{p}_1}H_\ast(\mathcal{M}_d(\ell^n))
\]
which implies that $[\mathcal{M}]=[\mathcal{M}']$ in the latter group.
\end{proof}

For even $d$ Example \ref{shape_space} is a bit different. This is related to the following lemma.

\begin{lemma}\label{2times}
 Let $d\geq 4$ be even and $\ell\in \R^n$ an ordered, generic length vector with $n\geq d+2$. Then
\begin{eqnarray*}
 2[\mathcal{M}_d(\ell^-)]&=&2[\mathcal{M}_d(\ell^+)] \,\,\,\in\,\,\,I^{\mathbf{p}_1}H_{\mathbf{d}^{n-1}_d}(\mathcal{M}_d(\ell)).
\end{eqnarray*}

\end{lemma}

\begin{proof}
Let
\begin{eqnarray*}
 \tilde{\mathcal{C}}_d(\ell)&=&\{ (x_1,\ldots,x_{n-1})\in \mathcal{C}_d(\ell)\,|\,x_{n-2},x_{n-1}\in S^1\}
\end{eqnarray*}
where $S^1\subset S^{d-2}$ is the standard inclusion using the first two coordinates. For generic $\ell$ this is a submanifold of $\mathcal{C}_d(\ell)$, and the projection $p:\tilde{\mathcal{C}}_d(\ell)\to S^1$ to the last coordinate has $-1\in S^1$ as a regular value. In particular, for $\theta\in S^1$ close to $-1$, we get that
\begin{eqnarray*}
 \tilde{\mathcal{C}}_d^\theta(\ell)&=&p^{-1}(\{\theta\})
\end{eqnarray*}
is diffeomorphic to $p^{-1}(\{-1\})$. Furthermore, this diffeomorphism can be chosen to be $\SO(d-2)$-equivariant, where $\SO(d-2)$ fixes the first two coordinates.

Fix $\theta_0\in S^1$ close enough to $-1$ so that this diffeomorphism exists. This gives rise to a $\SO(d-2)$-equivariant map
\[
 \tilde{H}\colon \tilde{\mathcal{C}}_d^{\theta_0}(\ell) \times [0,1] \to \mathcal{C}_d(\ell)
\]
which at time $t$ is inclusion to $p^{-1}(\{\theta_t\})$, where $\theta_t\in S^1$ starts at $\theta_0$ and ends at $-1$.

Define
\begin{eqnarray*}
 \mathcal{M}^{\theta_0}_d(\ell)&=&\tilde{\mathcal{C}}^{\theta_0}_d(\ell)/\SO(d-2),
\end{eqnarray*}
which is easily seen to be a pseudo-manifold of dimension $\mathbf{d}^{n-4}_d$ that represents an element $[\mathcal{M}^{\theta_0}_d(\ell)]\in I^{\mathbf{p}_1}H_{\mathbf{d}^{n-4}_d}(\mathcal{M}_d(\ell))$. Furthermore, we get the stratum preserving homotopy
\[
 H\colon \mathcal{M}_d^{\theta_0}(\ell)\times [0,1]\to \mathcal{M}_d(\ell)
\]
starting with the inclusion and ending with a surjection $H_1:\mathcal{M}_d^{\theta_0}(\ell)\to \mathcal{M}_d(\ell^-)$. For a generic set of points in $\mathcal{M}_d(\ell^-)$ the map $H_1$ is a $2:1$-map: if $x_{n-2}\not=\pm x_{n-1}$ for $x=[x_1,\ldots,x_{n-1}]\in \mathcal{M}_d(\ell^-)$ there are two points in $\mathcal{M}_d^{\theta_0}(\ell)$ send to $x$ coming from the $\Z/2$-action on $\mathcal{M}_d^{\theta_0}(\ell)$ which flips the second coordinate.

For even $d$ the $\Z/2$-action is orientation preserving, as we can rotate the second and the last coordinate by an angle of $\pi$, and for a generic point $x\in \mathcal{M}_d^{\theta_0}(\ell)$ there are $d-2$ elements which do not have a non-zero entry in the last coordinate. This means that $H_1$ is a degree $2$ map from $\mathcal{M}_d^{\theta_0}(\ell)$ to $\mathcal{M}_d(\ell^-)$, which implies that
\begin{eqnarray*}
 [\mathcal{M}_d^{\theta_0}(\ell)]&=& 2 [\mathcal{M}_d(\ell^-)].
\end{eqnarray*}
For $\eta$ near $+1\in S^1$ we can do a similar construction, showing that
\begin{eqnarray*}
 [\mathcal{M}_d^{\eta_0}(\ell)]&=& 2 [\mathcal{M}_d(\ell^+)].
\end{eqnarray*}
Let $J\subset S^1$ be the interval in the upper half plane with endpoints $\theta_0$ and $\eta_0$. Then $p^{-1}(J)\subset \tilde{\mathcal{C}}_d(\ell)$ is a cobordism between $\tilde{\mathcal{C}}_d^{\theta_0}(\ell)$ and $\tilde{\mathcal{C}}_d^{\eta_0}(\ell)$. Passing to the quotient under the $\SO(d-2)$ action shows that $[\mathcal{M}_d^{\theta_0}(\ell)]=[\mathcal{M}_d^{\eta_0}(\ell)]$.
\end{proof}

\begin{example}
 For $\ell^n$ from Example \ref{shape_space} with even $d$ we have $\mathcal{M}_d((\ell^n)^+)=\emptyset$, so over the rationals $I^{\mathbf{p}_1}H_{\mathbf{d}^{n-1}_d}(\mathcal{M}_d(\ell^n);\Q)=0$, and the rational intersection ring is trivial.

For $d=4$ we can actually show that
\begin{eqnarray*}
 I^{\mathbf{p}_1}H_{\mathbf{d}^{n-1}_4}(\mathcal{M}_4(\ell^n))&=& \Z/2\Z.
\end{eqnarray*}
In \cite[\S 5.2]{kebacl} it is shown that $H_{\mathbf{d}^{n-1}_4}(\Sigma^{n-1}_3)\cong \Z/2\Z$ and $\Sigma^{n-2}_3\approx \mathcal{M}_4(\ell^{n-1})$ represents the generator. The natural map
\[
 I^{\mathbf{p}_1}H_{\mathbf{d}^{n-1}_4}(\mathcal{M}_4(\ell^n)) \longrightarrow H_{\mathbf{d}^{n-1}_4}(\Sigma^{n-1}_3)
\]
is therefore surjective, and since the former is generated by $[\mathcal{M}_4(\ell^{n-1})]$ which has order $2$, it has to be an isomorphism.

We can now repeat the argument of Example \ref{shape_space} to show that
\begin{eqnarray*}
 I^{\mathbf{p}_s}H_{\mathbf{d}^{n-s}_4}(\mathcal{M}_4(\ell^n))&\cong & \Z/2\Z
\end{eqnarray*}
generated by $[\mathcal{M}_4(\ell^{n-s})]$ for $s=1,\ldots,n-5$, and the unreduced intersection ring satisfies
\begin{eqnarray*}
 I\!H^{3\ast}(\mathcal{M}_4(\ell))&\cong & \Z[X]/\langle X^{n-4},2X\rangle.
\end{eqnarray*}
We expect the analogous statement to hold also for $d\geq 6$ even, but the homology calculations for $\Sigma^{n-1}_{d-1}$ are somewhat more involved, compare \cite[\S 5]{kebacl}.
\end{example}

Recall the length vector $\ell_J$ for $J\subset \{1,\ldots,n-1\}$ defined in Section \ref{section_interring}. To simplify notation, we will write
\begin{eqnarray*}
 Y_i&=&[\mathcal{M}_d(\ell_{\{i\}})]\,\,\,\in\,\,\, I\!H^{d-1}(\mathcal{M}_d(\ell))
\end{eqnarray*}
for $i=1,\ldots,n-1$. Note that if $\{i,n\}$ is $\ell$-long, then $Y_i=0$. If $J\cup \{n\}$ is $\ell$-short, we also write
\begin{eqnarray*}
 Y_J&=& Y_{i_1}\cdots Y_{i_k}\,\,\,\in\,\,\, I\!H^{(d-1)k}(\mathcal{M}_d(\ell))
\end{eqnarray*}
for $J=\{i_1,\ldots,i_k\}$ with $i_1<\cdots < i_k$. By the properties of the intersection product we have $Y_J=\pm [\mathcal{M}_d(\ell_J)]$.

\begin{proposition}\label{linearly_ind}
 Let $\ell\in \R^n$ be a generic and ordered length vector with $n\geq d+2$, $d\geq 4$ and $k=a_1(\ell)$. Then $Y_1,\ldots,Y_k$ are linearly independent elements of $I\!H^{d-1}(\mathcal{M}_d(\ell))$.
\end{proposition}

\begin{proof}
 This follows directly from Lemma \ref{poin_dual} below, the proof of this Lemma does not require any further material from this section.
\end{proof}

%
%
%

\begin{theorem}\label{even_generators}
 Let $d\geq 4$ be even, $n\geq d+2$ and $\ell\in \R^n$ a generic, ordered, $d$-regular length vector. Then $I\!H^{d-1}(\mathcal{M}_d(\ell);\Q)$ is generated by $Y_1,\ldots,Y_{a_1(\ell)}$.
\end{theorem}

\begin{proof}
 The proof is by induction on $n$, and it starts with $n=d+2$. Since $\ell$ is $d$-regular, we have $a_2(\ell)=0$. We distinguish the cases $a_1(\ell)<n-1$ and $a_1(\ell)=n-1$. If $a_1(\ell)<n-1$, then $\mathcal{M}_d(\ell^+)=\emptyset$. By Lemma \ref{almost_calc}, $I^{\mathbf{p}_1}H_{\mathbf{d}^{n-1}_d}(\mathcal{M}^{m-1})$ has rank $1+a_1(\ell)$. Furthermore, $[\mathcal{M}_d(\ell^-)]$ is one of the generators. By Lemma \ref{2times} it represents $0$ in $I^{\mathbf{p}_1}H_{\mathbf{d}^{n-1}_d}(\mathcal{M}_d(\ell);\Q)$, and as the natural map $$I^{\mathbf{p}_1}H_{\mathbf{d}^{n-1}_d}(\mathcal{M}^{m-1};\Q)\to I^{\mathbf{p}_1}H_{\mathbf{d}^{n-1}_d}(\mathcal{M}_d(\ell);\Q)$$ is surjective, the result follows from Proposition \ref{linearly_ind}.

If $a_1(\ell)=n-1$, we use the original Morse-Bott function $F$, so that we have the minimum given by $\mathcal{M}_d(\ell^-)$ and the maximum given by $\mathcal{M}_d(\ell^+)$. As $\ell^+$ is $d$-regular, it has to be $(1,\ldots,1,d-1)\in \R^{d+1}$. Let $\mathcal{M}\subset \mathcal{M}_d(\ell)$ be the inverse image of a regular value slightly smaller than the maximum, so that $\mathcal{M}$ has the homotopy type of $\mathcal{M}_d(\ell)-\mathcal{M}_d(\ell^+)$. The standard Morse-Bott function on $\mathcal{M}_d(\ell^+)$ has two critical manifolds, $\mathcal{M}_d(\ell^{+-})$ and one absolute maximum $x\in \mathcal{M}_d(\ell^+)$. In $\mathcal{M}_d(\ell)-\mathcal{M}$ the minimum $\mathcal{M}_d(\ell^{+-})$ has index $d-1$, so it represents the generator of $I^{\mathbf{p}_1}H_{\mathbf{d}^{n-1}_d}(\mathcal{M}_d(\ell)-\{x\},\mathcal{M})$. However, with a construction as in the proof of Lemma \ref{2times} this element vanishes in $I^{\mathbf{p}_1}H_{\mathbf{d}^{n-1}_d}(\mathcal{M}_d(\ell),\mathcal{M};\Q)$. It follows 
that $I^{\mathbf{p}_1}H_{\mathbf{d}^{n-1}_d}(\mathcal{M}_d(\ell);\Q)\cong I^{\mathbf{p}_1}H_{\mathbf{d}^{n-1}_d}(\mathcal{M};\Q)$.

Recall that the Morse-Bott function $F$ has a critical manifold $S_J$ for every $J\subset\{1,\ldots,n-2\}$ such that $J\cup \{n\}$ is short, while $J\cup \{n-1,n\}$ is long. In particular, we get this for $\{i\}$ for all $i=1,\ldots,n-2$. Therefore
\begin{eqnarray*}
 I^{\mathbf{p}_1}H_{\mathbf{d}^{n-1}_d}(\mathcal{M};\Q)&\cong & \Q^{n-1}
\end{eqnarray*}
with one generator being $[\mathcal{M}_d(\ell^-)]$ and the others coming from the critical points. As $Y_1,\ldots,Y_{n-1}$ are linearly independent by Proposition \ref{linearly_ind}, they have to be a basis.

For the induction step, let $n>d+2$. We again use the Morse-Bott function $F$ and the subspace $\mathcal{M}$, which is a pseudomanifold with boundary. Let $\mathcal{N}=\mathcal{M}_d(\ell)-\mathcal{M}^o$, also a pseudomanifold with boundary $\partial\mathcal{N}=\partial\mathcal{M}$. If $a_1(\ell)<n-1$, then $\mathcal{N}=\mathcal{N}^{n-3}$. Also, the rank of $I^{\mathbf{p}_1}H_{\mathbf{d}^{n-1}_d}(\mathcal{M};\Q)$ is $1+a_1(\ell)$ by Lemma \ref{almost_calc}, and one of the generators is given by $[\mathcal{M}_d(\ell^-)]$. This generator dies in $I^{\mathbf{p}_1}H_{\mathbf{d}^{n-1}_d}(\mathcal{M}_d(\ell);\Q)$ by Lemma \ref{2times}. Hence the rank of this group is $a_1(\ell)$, and the result follows by Proposition \ref{linearly_ind}.

If $a_1(\ell)=n-1$, the inclusion $\mathcal{M}_d(\ell^+)\subset \mathcal{N}$ is a homotopy equivalence which induces an isomorphism on intersection homology if we add $\mathbf{p}_1$ to the perversity.

Recall the top perversity $\mathbf{t}_n=(0,\mathbf{c}^n_{d,2}-2,\ldots,\mathbf{c}^n_{d,d-2}-2)$. It is easy to check that $\mathbf{t}_n-\mathbf{p}_1=\mathbf{t}_{n-1}$. Using Lefschetz duality for pseudomanifolds with boundary, we get
\begin{eqnarray*}
 I^{\mathbf{p}_1}H_\ast(\mathcal{M}_d(\ell),\mathcal{M};\Q)&\cong & I^{\mathbf{p}_1}H_\ast(\mathcal{N},\partial\mathcal{N};\Q)\\
&\cong & I^{\mathbf{t}_n-\mathbf{p}_1}H_{\mathbf{d}^n_d-\ast}(\mathcal{N};\Q)\\
&\cong & I^{\mathbf{t}_{n-1}}H_{\mathbf{d}^n_d-\ast}(\mathcal{N};\Q)\\
&\cong & I^{\mathbf{t}_{n-1}-\mathbf{p}_1}H_{\mathbf{d}^n_d-\ast}(\mathcal{M}_d(\ell^+);\Q)\\
&\cong & I^{\mathbf{p}_1}H_{\ast-(d-1)}(\mathcal{M}_d(\ell^+);\Q).
\end{eqnarray*}
The long exact sequence of $(\mathcal{M}_d(\ell),\mathcal{M})$ thus turns into
\begin{multline*}
 \cdots \longrightarrow I^{\mathbf{p}_1}H_{\mathbf{d}^{n-2}_d+1}(\mathcal{M}_d(\ell^+);\Q) \longrightarrow I^{\mathbf{p}_1}H_{\mathbf{d}^{n-1}_d}(\mathcal{M};\Q) \longrightarrow \\
I^{\mathbf{p}_1}H_{\mathbf{d}^{n-1}_d}(\mathcal{M}_d(\ell);\Q) \longrightarrow I^{\mathbf{p}_1}H_{\mathbf{d}^{n-2}_d}(\mathcal{M}_d(\ell^+);\Q) \longrightarrow I^{\mathbf{p}_1}H_{\mathbf{d}^{n-1}_d-1}(\mathcal{M};\Q)
\end{multline*}
and the last term is $0$ by Lemma \ref{almost_calc}.

By induction, the rank of $I^{\mathbf{p}_1}H_{\mathbf{d}^{n-2}_d}(\mathcal{M}_d(\ell^+);\Q)$ is $a_1(\ell^+)$. Note that  $\{i,n-1\}$ is $\ell^+$-short if and only if $\{i,n-1,n\}$ is $\ell$-short. The rank of $I^{\mathbf{p}_1}H_{\mathbf{d}^{n-1}_d}(\mathcal{M};\Q)$ is $1+c_{n-4}$, where $c_{n-4}$ is the number of critical points of index $(n-4)(d-1)$ of $F$. These critical points correspond to sets $\{i\}$ with $\{i,n\}$ $\ell$-short, but $\{i,n-1,n\}$ $\ell$-long. Hence $1+c_{n-4}+a_1(\ell^+)=a_1(\ell)=n-1$. Therefore the rank of $I^{\mathbf{p}_1}H_{\mathbf{d}^{n-1}_d}(\mathcal{M}_d(\ell))$ can be at most $n-1$. By Proposition \ref{linearly_ind}, it has to be $n-1$.
\end{proof}

\section{Calculation of the intersection ring}

Whenever $J\subset \{1,\ldots,n-1\}$ satisfies $J\cup \{n\}$ is $\ell$-short, we have an element $Y_J\in I\!H^{|J|(d-1)}(\mathcal{M}_d(\ell))$. We want these to be linearly independent. In order to do this, we construct explicitly the Poincar\'e dual of $\mathcal{M}_d(\ell_J)$.

\begin{lemma}\label{poin_dual}
 Let $n\geq d+1\geq 5$ and $\ell\in \R^n$ be a generic, ordered, $d$-regular length vector, and $J\subset \{1,\ldots,n-1\}$ with $J\cup \{n\}$ being $\ell$-short. Then there exists an element $Z_J\in I^\mathbf{0}H_{|J|(d-1)}(\mathcal{M}_d(\ell))$ with
\begin{eqnarray*}
 Z_J \cdot Y_K &=&\left\{ \begin{array}{cl} 1 & J=K \\ 0 & \mbox{\rm else} \end{array}\right.
\end{eqnarray*}

\end{lemma}

\begin{proof}
 We want to construct an embedding $(S^{d-1})^{|J|} \to \mathcal{M}_d(\ell)-\mathcal{N}_{d-2}(\ell)$ which has empty intersection with $\mathcal{M}_d(\ell_K)$ for every $K\subset \{1,\ldots,n-1\}$ with $|K|=|J|$, $K\cup \{n\}$ $\ell$-short and $K\not=J$, and which intersects $\mathcal{M}_d(\ell_J)$ transversely in exactly one point. By the standard properties of the intersection pairing this will prove the Lemma.

After reordering, we can assume that $J=\{k+1,\ldots,n-1\}$. By $d$-regularity, we have $k\geq d\geq 4$. By genericity we can assume that $\ell_1<\ell_2<\cdots <\ell_k<\ell_n$. We have
\begin{eqnarray*}
 \ell_1+\cdots + \ell_{k-1}-\ell_k & > & \ell_{k+1}+\cdots + \ell_{n-1} - \ell_n
\end{eqnarray*}
as
\begin{eqnarray*}
 \ell_1+\cdots + \ell_{k-1} + \ell_n & > & \ell_{k+1}+\cdots + \ell_{n-1} + \ell_k
\end{eqnarray*}
which is true because $J\cup \{n\}$ is $\ell$-short and $\ell_n>\ell_k$.

Let $B\subset \R^d$ be the closed ball of radius $\ell_{k+1}+\cdots+\ell_{n-1}$ centered at $-\ell_n \cdot e_1$. We want to find a map $f\colon B\to (S^{d-1})^k$ with $A\circ f = {\rm id}_B$, where $A\colon(S^{d-1})^k \to \R^d$ is given by
\begin{eqnarray*}
 A(x_1,\ldots,x_k)&=& \sum_{i=1}^k \ell_i x_i.
\end{eqnarray*}
The map $F\colon(S^{d-1})^{|J|}\to \mathcal{C}_d(\ell)$ given by
\begin{eqnarray}\label{map_needed}
 F(x_{k+1},\ldots,x_{n-1})&=&\left(f\left(-\ell_n\cdot e_1-\sum_{i=k+1}^{n-1}\ell_ix_i\right),x_{k+1},\ldots,x_{n-1}\right)
\end{eqnarray}
is then nearly the map that we need.

In order to construct $f$, let us first consider the case $d=2$. We begin by constructing a `Snake charmer', a map $\gamma:[0,1] \to (S^{d-1})^k$ such that 
\begin{eqnarray}\label{snakes_charm}
A\circ \gamma (t) &=& (tC_1 + (1-t)C_2)\cdot e_1
\end{eqnarray}
where
\[
 -(\ell_1+\cdots +\ell_k) \,\,\, < \,\,\, C_1 \,\,\, < \,\,\, -(\ell_{k+1}+\cdots + \ell_n
\]
and
\[
 \ell_{k+1} +\cdots + \ell_{n-1} - \ell_n \,\,\, < \,\,\, C_2 \,\,\, < \,\,\, \ell_1+\cdots + \ell_{k-1} - \ell_k.
\]
To do this, we can start in the position $(e_1,\ldots,e_1,-e_1)\in (S^{d-1})^k$ and then start to rotate the $(k-1)$-th coordinate counterclockwise into the upper half plane. At the same time, the $k$-th coordinate rotates counterclockwise so that the robot arm $\ell_{k-1}x_{k-1}+\ell_kx_k$ remains on the first axis. We continue this until the $(k-1)$-th coordinate is nearly rotated to $-e_1$. After that we rotate the $(k-2)$-th coordinate counterclockwise into the upper half plane and rotate the $k$-th and $(k-1)$-th coordinate so that the robot arm consisting of the last three coordinates ramins on the first axis. Here the $k$-th and $(k-1)$-th coordinates are rotated by the same amount, so that these two links remain stiff. When the $(k-2)$-th coordinate nearly reached the $-e_1$ position, we start to rotate the $(k-3)$-th coordinate counterclockwise, using the last three coordinates to keep the robot arm on the first axis. We continue this until all coordinates are near $-e_1$. After reparametrisation, 
we have the desired snake charmer.

We actually do not want any of the links to point to $e_1$. So rather than starting with $(e_1,\ldots,e_1,-e_1)$ we start at a position $(x^\ast,\ldots,x^\ast,y^\ast)\in (S^{d-1})^k$ with $x^\ast$ in the upper half plane so close to $e_1$ that the resulting snake charmer still satisfies (\ref{snakes_charm}).
\begin{figure}[ht]
\begin{center}
\includegraphics[height=6cm,width=6cm]{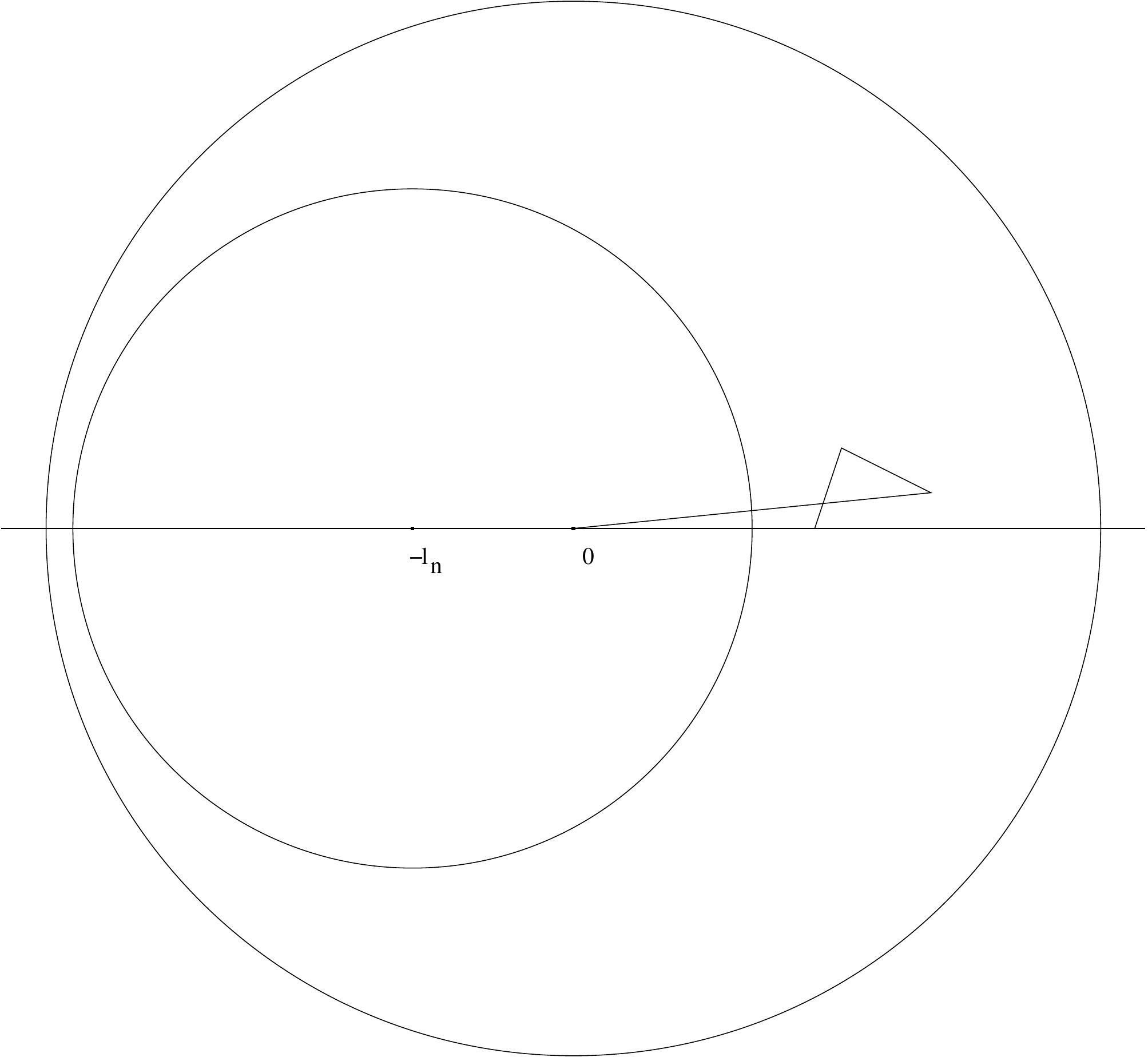}
\caption{\label{fig_snake}Snake charming along the $x$-axis.}
\end{center}
\end{figure}
Figure \ref{fig_snake} indicates this map. The outer circle has radius $\ell_1+\cdots+\ell_k$, so represents all the points the robot arm in the first $k$ coordinates could reach, and the inner circle bounds $B$.

This construction also provides numbers $0=t_0<t_1<\cdots < t_{k-1} = 1$ such that on the interval $[t_i,t_{i+1}]$ the first $k-(i+2)$ coordinates are fixed, the $k-(i+1)$-th coordinate rotates from a position near $e_1$ to a position near $-e_1$, and all other coordinates have negative scalar product with $e_1$. We refer to $A\circ \gamma$ as the robot arm $R$.

We now want to extend $\gamma$ to a map $\Gamma\colon [0,1]\times [-1,1] \to (S^{d-1})^k$ so that $A\circ \Gamma$ restricts to a homeomorphism of a closed subset $\tilde{B}\subset [0,1]\times [-1,1]$ to $B\subset \R^2$. Consider the interval $[t_{k-2},1]\subset [0,1]$. On this interval we want to rotate the robot arm $R$ into the plane. To do this rotation, consider $[-1,1]\subset S^1$ as those points with first coordinate non-positive (and $0$ corresponding to $-e_1\in S^1$). We need a map $h\colon [-1,1] \to \SO(2)$ with $h(x)\cdot (-e_1) = x$, which is no problem as $\SO(2)$ can be identified with $S^1$ via the action. We could then define $\Gamma$ on $[t_{k-2},1]\times [-1,1]$ by $\Gamma(t,x)=h(x)\cdot \gamma(t)$. Doing this will make it difficult to extend $\Gamma$ to $[0,1]\times [-1,1]$, and also the first coordinate can be rotated to $e_1$.

Note that there is a unique point $t^\ast\in (t_{k-2},1)$ with the first coordinate of $\gamma(t^\ast)$ equal to $e_2$. From this point on we do not rotate the first coordinate by the same amount as the other coordinates, and by the time $t=t_{k-2}$, only the coordinates $2,\ldots,k$ will rotate via $h$, the first coordinate will be fixed in the position $x^\ast$. This way we can assure that the first coordinate is always different from $e_1$. The other coordinates are also different from $e_1$, as they start with negative scalar product with $e_1$ and rotate by at most an angle of $\pi/4$.
\begin{figure}[ht]
\begin{center}
\includegraphics[height=4cm,width=6cm]{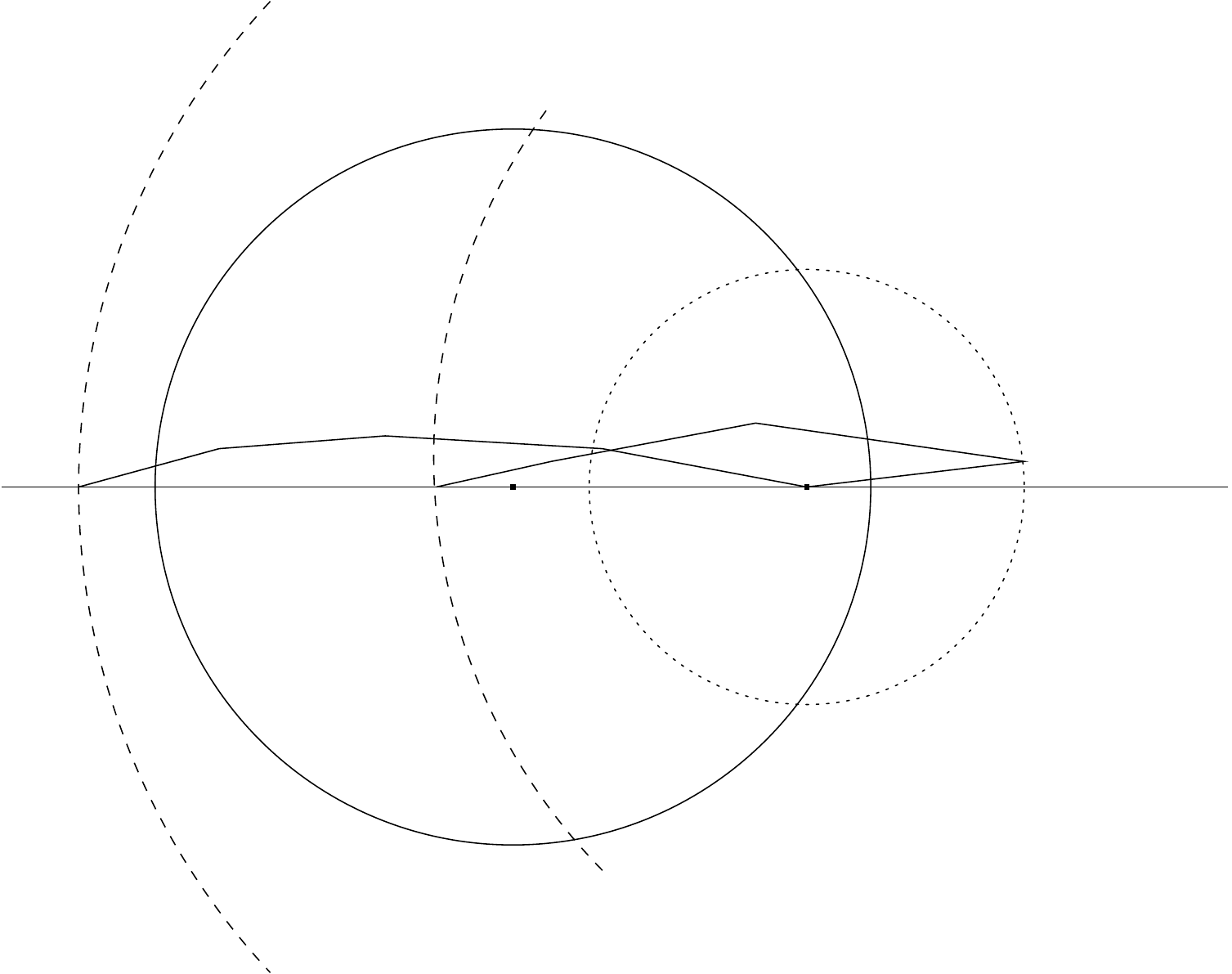}
\caption{\label{fig_snake_rot}Rotating along the origin and along $x^\ast$.}
\end{center}
\end{figure}
Figure \ref{fig_snake_rot} indicates the movement of the robot arm near $t=1$ and near $t=t_{n-2}$. The solid circle bounds $B$, the dotted circle indicates the movements of the first coordinate of the robot arm $R$. By elementary geometry, the map $A\circ \Gamma$ maps a closed subset of $[t_{k-2},1]\times [-1,1]$ homeomorphically onto $B\cap A\circ \Gamma([t_{k-2},1]\times [-1,1])$.

We now extend $\Gamma$ to $[t_{k-3},t_{k-2}]\times [-1,1]$ in basically the same way. We first rotate all of the coordinates $2,\ldots,k$, eventually rotating the second coordinate less, until it is no longer rotated at time $t=t_{k-3}$.

After finitely many steps we have a map $\Gamma\colon [0,1]\times [-1,1] \to (S^{d-1})^k$ and a closed subset $\tilde{B}\subset [0,1]\times [-1,1]$ such that $A\circ \Gamma$ maps $\tilde{B}$ homeomorphically onto $B$.

Note that we were assuming $d=2$, and we need the analogous result for $d\geq 4$. In order to do this, replace $[-1,1]$ by $D^{d-1}=\{(x_1,\ldots,x_d)\in S^{d-1}\,|\,x_1\leq 0\}$. The map $p\colon \SO(d) \to S^{d-1}$ given by $p(A)=A\cdot(-e_1)$ is a fiber bundle, so there exists a section $h\colon D^{d-1} \to \SO(d)$ which then can be used to define $\Gamma$.

The resulting map $F\colon (S^{d-1})^{|J|}\to \mathcal{C}_d(\ell)$ given by (\ref{map_needed}) then intersects $\mathcal{C}_d(\ell_J)$ transversely in exactly one point, and has empty intersection with $\mathcal{C}_d(\ell_K)$ for all $K\subset\{1,\ldots,n-1\}$ with $|K|=|J|$ and $K\not= J$. The point of intersection is given by $F(e_1,\ldots,e_1)$.

Consider the induced map $\bar{F}\colon  (S^{d-1})^{|J|}\to \mathcal{M}_d(\ell)$. With the current construction we have that $\gamma$ has image in $(S^1)^k$. This means that the image of $\bar{F}$ will intersect lower strata of $\mathcal{M}_d(\ell)$. However, if we modify the robot arm $\gamma$ slightly by using the higher dimensions, the map $A\circ \Gamma$ will remain injective and we can repeat the construction so that the image of $\bar{F}$ is in $\mathcal{M}_d(\ell)-\mathcal{N}_{d-2}(\ell)$. This is using that $k\geq d$. The transverse intersection of $F$ with $\mathcal{C}_d(\ell_J)$ induces a transverse intersection of $\bar{F}$ with $\mathcal{M}_d(\ell_J)$ in exactly one point, so that $\bar{F}$ induces an element $Z_J\in I^\mathbf{0}H_{|J|(d-1)}(\mathcal{M}_d(\ell))$ which has the desired properties.
\end{proof}

The fact that the $Z_J$ can be defined with $\mathbf{0}$-perversity means that the $Y_J$ remain linearly independent in ordinary homology. The reduced intersection ring for even $d$ can now be determined.

\begin{definition}
 Let $\Delta$ be a finite abstract simplicial complex, that is, a collection of subsets of a set $\{x_1,\ldots,x_k\}$ which is closed under subsets. The \em exterior face ring \em $\Lambda_\K[\Delta]$ over the commutative ring $\K$ is the quotient of the exterior algebra $\Lambda_\K[X_1,\ldots,X_k]$ by the ideal generated by elements $X_{i_1}\cdots X_{i_m}$ where $\{x_{i_1},\ldots,x_{i_m}\}\notin \Delta$.
\end{definition}

Note that for a length vector $\ell\in \R^n$ the collection
\begin{eqnarray*}
 \mathcal{S}_\cdot(\ell)&=&\bigcup_{k=1}^{n-3} \mathcal{S}_k(\ell)
\end{eqnarray*}
is an abstract simplicial complex with vertex set $\mathcal{S}_1(\ell)$.

\begin{theorem}\label{the_ext_face_ring}
 Let $d\geq 4$ be even, $\ell\in \R^n$ a generic, $d$-regular length vector. Then the reduced intersection ring of $\mathcal{M}_d(\ell)$ with rational coefficients is the exterior face ring
\begin{eqnarray*}
 I\!H^{(d-1)\ast}(\mathcal{M}_d(\ell);\Q)&\cong & \Lambda_\Q[\mathcal{S}_\cdot(\ell)].
\end{eqnarray*}

\end{theorem}

\begin{proof}
We can assume that $\ell$ is ordered. By Theorem \ref{even_generators} the reduced intersection ring is generated by $Y_1,\ldots,Y_k$ with $k=a_1(\ell)=|\mathcal{S}_1(\ell)|$. As $d$ is even, we get that
\begin{eqnarray*}
 Y_i\cdot Y_j &=& - Y_j \cdot Y_i
\end{eqnarray*}
for all $i,j\leq k$ by \cite[\S 2.4]{gormac}. Therefore 
\begin{eqnarray*}
I\!H^{(d-1)\ast}(\mathcal{M}_d(\ell);\Q)&\cong & \Lambda_\Q[Y_1,\ldots,Y_k]/I
\end{eqnarray*}
for some ideal $I$. If $K\subset \{1,\ldots,n-1\}$ is such that $K\cup \{n\}$ is $\ell$-long, then $\mathcal{M}_d(\ell_K)=\emptyset$, so $Y_K=0$, which means that $Y_K\in I$. It remains to show that there are no other relations in $I$. But by Lemma \ref{poin_dual} all $Y_J$ with $J\in \mathcal{S}_\cdot(\ell)$ are linearly independent, which means that $I$ is indeed generated by $Y_K$ with $K\cup\{n\}$ long.
\end{proof}

\begin{proof}[Proof of Theorem \ref{main_theorem}]
We can assume that both $\ell$ and $\ell'$ are ordered. The reduced intersection rings are homeomorphism invariants by \cite[\S 5]{gorma2}, so $\Lambda_\Q[\mathcal{S}_\cdot(\ell)]\cong \Lambda_\Q[\mathcal{S}_\cdot(\ell')]$. By \cite[Exercise 5.12]{brugub} there is an isomorphism of simplicial complexes $\mathcal{S}_\cdot(\ell)\cong \mathcal{S}_\cdot(\ell')$. But as $\ell$ and $\ell'$ are ordered, it follows that $\mathcal{S}_\cdot(\ell)= \mathcal{S}_\cdot(\ell')$, compare \cite[Lemma 3]{fahasc}. Therefore $\ell$ and $\ell'$ are in the same chamber.
\end{proof}

\begin{remark}
 If a length vector $\ell$ is $d$-regular, it is also $k$-regular for all $k<d$. By Theorem \ref{the_ext_face_ring}, the reduced intersection ring of $\mathcal{M}_d(\ell)$ does not depend on $d$ in the sense that it is isomorphic to the reduced intersection ring of $\mathcal{M}_k(\ell)$ for all even $k$ with $4\leq k \leq d$. One approach to extend Theorem \ref{main_theorem} would be to try to get a similar independence of $d$ in the odd case all the way to $k=3$. The intersection ring could then be determined using \cite{hauknu}. Notice that in \cite{fahasc} cohomology with $\Z/2$ coefficients was needed, so getting the analogous statement with rational coefficients would not be enough.

The statement of Theorem \ref{main_theorem} is true in the case $d=3$, but the condition of $3$-regularity can be replaced by $n>4$. So in fact there are only two length vectors which have homeomorphic linkage spaces but are not in the same chamber up to permutation. For $d=4$ one may therefore ask whether $4$-regularity can be replaced by $n>5$, or if one can give an example of length vectors from different chambers with $n>5$ with homeomorphic linkage spaces.
\end{remark}

\end{document}